%% file: formal_mes_paper.tex
\documentclass[12pt,reqno,a4paper]{amsart}  
\usepackage{amsmath,amsthm,amssymb,amscd,bbm,etoolbox,faktor,float,geometry,mathdots,mathtools,enumerate,shuffle,xcolor}
\usepackage[pdfborder={0 0 0}]{hyperref}
\usepackage{stmaryrd}
\usepackage{tikz,tikz-cd}
\mathtoolsset{showonlyrefs}
\usetikzlibrary{decorations,decorations.pathreplacing,angles,shapes,calc,arrows, bending, positioning,quotes}

\patchcmd{\maketitle}
  {\ifx\@empty\@dedicatory}
  {\ifx\@empty\@dedicatory\vspace{1em}\centering With an appendix by Nils Matthes \par\vspace{1em}}
  {}{}

\title{Formal multiple Eisenstein series and their derivations}
\author{Henrik Bachmann}
\address{Graduate School of Mathematics, Nagoya University, Nagoya, Japan.}
\email{henrik.bachmann@math.nagoya-u.ac.jp}

\author{Jan-Willem van Ittersum}
\address{Max-Planck-Institut f\"ur Mathematik, Vivatsgasse 7, 53111 Bonn, Germany.
}
\curraddr{Department of Mathematics and Computer Science, University of Cologne,  Weyertal 86-90\\
\indent 50931 Cologne, Germany}
\email{j.w.ittersum@uni-koeln.de}

\address{Department of Mathematical Sciences, University of Copenhagen, Universitetsparken 5\\ \indent  2100 Copenhagen Ø, Denmark}
\email{nils.oliver.matthes@gmail.com}


\date{\today}

\newtheorem{theorem}{Theorem}[section]
\newtheorem{proposition}[theorem]{Proposition}

\newtheorem{lemma}[theorem]{Lemma}
\newtheorem{corollary}[theorem]{Corollary}
\newtheorem{conjecture}[theorem]{Conjecture}
\numberwithin{equation}{section}
\theoremstyle{definition}
\newtheorem{definition}[theorem]{Definition}
\newtheorem{example}[theorem]{Example}
\newtheorem{remark}[theorem]{Remark}

\newcommand{\Z}{\mathbb{Z}}
\newcommand{\Q}{\mathbb{Q}}
\newcommand{\R}{\mathbb{R}}
\newcommand{\C}{\mathbb{C}}
\newcommand{\Ha}{\mathbb{H}}

\newcommand{\mz}{\mathcal{Z}}
\newcommand{\mf}{\mathcal{M}}
\newcommand{\qmf}{\widetilde{\mf}}
\newcommand{\fmes}{\mathcal{G}^{\!f}} 
\newcommand{\fmesz}{\fil{\lwt}{0}\fmes}
\newcommand{\ames}{\mathcal{E}} % The "analytic" MES defined as lattice sums  
\newcommand{\fames}{\mathcal{E}^{\!f}} 
\newcommand{\fmz}{\mathcal{Z}^f}
\newcommand{\fqmf}{\widetilde{\mathcal{M}}^{\!f}}
\newcommand{\fmf}{\mathcal{M}^{\!f}}
\newcommand{\fs}{\mathcal{S}^{\!f}}
 \newcommand{\fH}{\mathcal{G}_0^f}
\newcommand{\h}{\mathfrak H}
\newcommand{\im}{\operatorname{im}}

\newcommand{\D}{D}
\newcommand{\W}{W}

\newcommand{\proj}{\pi} % Projection to formal MZV
\newcommand{\df}{\coloneqq}
\newcommand{\deer}{\delta}
\newcommand{\ttt}{\omega_3} % Derivation of weigt -3. Former "t"

\newcommand{\dphia}{\Theta^{\varphi,a}}
\newcommand{\dlefta}{\Theta^{[a}}
\newcommand{\drighta}{\Theta^{a]}}
\newcommand{\tA}{\widetilde{\mathcal{A}}}
\newcommand{\tD}{\widetilde{\D}}
\newcommand{\tW}{\widetilde{\W}}
\newcommand{\tdeer}{\widetilde{\deer}}
\newcommand{\tDd}{\tD_\diamond}
\newcommand{\tWd}{\tW_\diamond}
\newcommand{\tdeerd}{\tdeer_\diamond}

\newcommand{\gf}{G^{\!f}\!}
\newcommand{\fzeta}{\zeta^f\!}
\newcommand{\fDelta}{\Delta^\mathfrak{f}}
\DeclareRobustCommand{\eis}{e}

\newcommand{\I}{\mathcal{I}} % The ideal gen. by sigma(w)-w

\newcommand{\fil}[2]{\operatorname{Fil}^{\mathrm{#1}}_{#2}}
\newcommand{\lwt}{\operatorname{lwt}}

\newcommand{\dep}{\operatorname{dep}}

\newcommand{\wt}{\operatorname{wt}}
\newcommand{\sltwo}{\mathfrak{sl}_2}
\newcommand{\sltwoz}{\mathrm{SL}_2(\mathbb{Z})}

\newcommand{\genA}{\mathfrak{A}}
\newcommand{\genG}{\mathfrak{G}}

\newcommand{\ind}{\mathbf{1}}
\newcommand{\kd}[2]{\ind_{\substack{#1\\#2}} }

\DeclareRobustCommand{\ai}{\genfrac{[}{]}{0pt}{}}
\DeclareRobustCommand{\bi}{\genfrac{(}{)}{0pt}{}}
\DeclareRobustCommand{\ebi}{\genfrac{}{}{0pt}{}}

\newcommand{\one}{{\bf 1}}
\newcommand{\A}{\mathcal{A}}
\newcommand{\QA}{\Q\langle \A \rangle}
\newcommand{\B}{\mathcal{B}}
\newcommand{\QB}{\Q\langle \B \rangle}
\newcommand{\QLZ}{\Q\langle \LL_z \rangle}
\newcommand{\LL}{\mathcal{L}}
\newcommand{\QL}{\Q\langle \LL \rangle}
\newcommand{\qsh}{\ast_\diamond}

\newcommand{\X}{\underline{X}}
\newcommand{\Y}{\underline{Y}}

\newcommand*\circled[1]{\tikz[baseline=(char.base)]{
    \node[shape=circle,draw,inner sep=0.5pt] (char) {$#1$};}}
\newcommand{\osh}{\raisebox{1.3pt}{\;\circled{\scalebox{0.53}{$\shuffle$}}}\;}
\newcommand{\oshh}{\raisebox{1.3pt}{\circled{\scalebox{0.53}{$\shuffle$}}}}

\renewcommand{\=}{\: =\: }

\renewcommand{\Im}{\mathrm{Im}\,}
\newcommand{\HH}{\mathbb{H}}
\newcommand{\qmfpi}{\widetilde{\mathbb{M}}}
\newcommand{\mfpi}{\mathbb{M}}
\newcommand{\cusppi}{\mathbb{S}}

\def\cB{\mathcal B}
\def\cG{\mathcal G}
\def\cM{\mathcal M}

\def\bQ{\mathbb Q}

\def\fH{\mathfrak{H}}

\def\sy{z}

\DeclareMathOperator{\Hom}{Hom}

\begin{document}

\maketitle
\vspace{-0.92cm}
%\vspace{-29pt}
\begin{center}
{\footnotesize (\emph{With an appendix by} NILS MATTHES)}
\end{center}

\begin{abstract} 
 We introduce the algebra of formal multiple Eisenstein series and study its derivations. This algebra is motivated by the classical multiple Eisenstein series,  introduced by Gangl--Kaneko--Zagier as a hybrid of classical Eisenstein series and multiple zeta values. 
In depth one, we obtain formal versions of the Eisenstein series satisfying the same algebraic relations as the classical Eisenstein series. In particular, they generate an algebra whose elements we call formal quasimodular forms.
We show that the algebra of formal multiple Eisenstein series is an $\sltwo$-algebra by formalizing the usual derivations for quasimodular forms and extending them naturally to the whole algebra. 
Additionally, we introduce some families of derivations for general quasi-shuffle algebras, providing a broader context for these derivations.
Further, we prove that a quotient of this algebra is isomorphic to the algebra of formal multiple zeta values. This gives a novel and purely formal approach to classical (quasi)modular forms and builds a new link between (formal) multiple zeta values and modular forms.
\end{abstract}

\section{Introduction}

The purpose of this note is to propose a notion of formal multiple Eisenstein series, and explain their intimate connection to formal multiple zeta values and the $\sltwo$-algebra of formal (quasi)modular forms. Before outlining the surprisingly simple steps to obtain the formal space in Section~\ref{sec:formal}, we first recall the `concrete' objects we aim to `formalize', i.e., the notions of multiple Eisenstein series, multiple zeta values and (quasi)modular forms. 

\subsection{Multiple zeta values}
Multiple zeta values, which are defined for integers $r\geq 1$ and $k_1\geq 2, k_2,\dots,k_r \geq 1$ by
\begin{align}\label{eq:defmzv}
	\zeta(k_1,\ldots,k_r) \df \sum_{m_1>\cdots>m_r>0} \frac{1}{m_1^{k_1}\cdots m_r^{k_r}}
\end{align}
are subject to many relations. Denote the $\Q$-algebra of all multiple zeta values by $\mz$. 
Conjecturally, the \emph{extended double shuffle relations} of multiple zeta values provide all algebraic relations among multiple zeta values~\cite{IKZ}. These relations are obtained (after possible regularization) from the two ways of expressing the product of multiple zeta values---the `usual' (stuffle) product of real numbers, and a (shuffle) product from the iterated integral representation of multiple zeta values---which both can be interpreted as quasi-shuffle products~\cite{H}. 

First, we recall the usual algebraic setup for multiple zeta values, following~\cite{IKZ} and~\cite{H}, and introduce formal multiple zeta values. Consider the alphabet $\LL_z=\{z_k \mid k\geq 1\}$ and let $\h^1 = \QLZ$ be the free algebra over $\LL_z\mspace{1mu}$. Define a product~$\diamond$ on $\Q \LL_z$ by $z_{i} \diamond z_{j} = z_{i+j}$ for all $i,j\geq 1$. The corresponding quasi-shuffle product~\eqref{eq:qshdef} $\ast = \qsh$ is usually called the \emph{stuffle product} and $(\h^1,\ast)$ is a commutative $\Q$-algebra. Let $(\h^0,\ast)$ be the subalgebra of $(\h^1,\ast)$ generated by all words not starting in $z_1\mspace{1mu}$. The linear map defined on generators by 
\vspace{-0.1cm}
\begin{align}\begin{split} \label{eq:mapzeta}
\zeta: \h^0&\longrightarrow \mz\\ 
z_{k_1}\cdots z_{k_r}&\longmapsto \zeta(k_1,\ldots,k_r)
\end{split}
\end{align}
is an algebra homomorphism from $(\h^0,\ast)$ to $\mz$.
This homomorphism can be extended to a homomorphism $\zeta^\ast:\h^1\to\mz[T]$, that is, we obtain elements $\zeta^\ast(k_1,\dots,k_r) \in \R[T]$ for all $k_1,\dots,k_r \geq 1$ called the \emph{stuffle regularized multiple zeta values} (see~\cite{IKZ}). In case $k_1 \geq 2$, these coincide with the multiple zeta values~\eqref{eq:defmzv}. 

The second product is given by the shuffle product for which we consider the alphabet given by the two letters $\LL_{xy}=\{x,y\}$ and write $\h=\Q\langle \LL_{xy} \rangle$. Define the product~$\diamond$ by $a\diamond b=0$ for all $a,b\in\Q \LL_{xy}\mspace{1mu}$. Denote the corresponding quasi-shuffle product~$\qsh$ by $\shuffle$. Through the identification $z_k = x^{k-1}y$ we can view $\h^1$ and $\h^0$ as subalgebras of $(\h,\shuffle)$.   Due to the iterated integral expression of multiple zeta values, one obtains that the map~\eqref{eq:mapzeta} gives an algebra homomorphism from $(\h^0,\shuffle)$ to $\mz$. There is also a unique extension of the map $\zeta$ to an algebra homomorphism $\zeta^\shuffle: (\h^1,\shuffle) \rightarrow \mz[T]$. 

The two regularizations~$\zeta^
\ast$ and~$\zeta^\shuffle$ differ and their difference can be described explicitly (see~\cite[Theorem~1]{IKZ}). The comparison of shuffle- and stuffle-regularized multiple zeta values, leads to the \emph{extended double shuffle relations}. The space of \emph{formal multiple zeta values} is then defined as $\h^1$ modulo the extended double shuffle relations (cf., Definition~\ref{def:regdsh}). The class of a word $z_{k_1}\cdots z_{k_r}$ is denoted by $\fzeta(k_1,\dots,k_r)$. We will give a new definition of formal multiple zeta values (Definition~\ref{def:formalmzv}) and then show that they coincide with the one mentioned here (Theorem~\ref{thm:formalmzvdsh}).

\subsection{Multiple Eisenstein series} Multiple zeta values and (quasi)modular forms are connected in various ways. For example, in the case $r=1$, a multiple zeta value is a Riemann zeta value, which is the constant term of the Eisenstein series. The Eisenstein series of weight $k\geq 2$ is given for $\tau \in \Ha = \{ {\tau \in \C \mid }\operatorname{Im}(\tau) > 0 \}$ by 
\begin{align}\label{eq:eis}
    \mathbb{G}_k(\tau) \df \zeta(k) + \frac{(-2\pi i)^k}{(k-1)!} \sum_{m,n\geq 1} n^{k-1} q^{mn} \qquad \qquad (q=e^{2\pi i \tau}).
\end{align}
For even $k\geq 2$ these series are (quasi)modular forms for the full modular group. In~\cite{GKZ} the authors defined double Eisenstein series, which have double zeta values (\eqref{eq:defmzv} in the case $r=2$) as their constant terms, and which can be seen as a natural depth two version of Eisenstein series. This construction was generalized by the first author in~\cite{Ba0}. 
Given a \emph{depth} $r\geq 1$ and integers $k_1,\dots,k_r \geq 2$ the \emph{multiple Eisenstein series} are defined\footnote{In the case $k_1=2$ one needs to use Eisenstein summation. See~\cite{Ba3} for details.} for $\tau \in \Ha$ by
\begin{align}\label{eq:defmes}
\mathbb{G}_{k_1,\dots,k_r}(\tau) \df \sum_{\substack{\lambda_1 \succ \dots \succ \lambda_r \succ 0\\ \lambda_i \in \Z \tau + \Z}} \frac{1}{\lambda_1^{k_1} \cdots \lambda_r^{k_r}}   \,,
\end{align}
where the order $\succ$ on the lattice $\Z \tau + \Z$ is defined by $m_1 \tau + n_1 \succ m_2 \tau + n_2$ iff $m_1 > m_2$ or $m_1 = m_2 \wedge n_1 > n_2$. The multiple Eisenstein series are holomorphic functions, i.e. elements in~$\mathcal{O}(\Ha)$.
Since $\mathbb{G}_{k_1,\dots,k_r}(\tau + 1) = \mathbb{G}_{k_1,\dots,k_r}(\tau)$, which can be obtained directly by the above definition, the multiple Eisenstein series possess a Fourier expansion 
\begin{align}\label{eq:mesfourier}
    \mathbb{G}_{k_1,\dots,k_r}(\tau) = \zeta(k_1,\dots,k_r) \,+\, \sum_{n\geq 1} a_n \,q^n \qquad (a_n \in \mz[2\pi i]), 
\end{align}
 which was calculated in depth $r=2$ in~\cite{GKZ} and for arbitrary depth by the first author in~\cite{Ba0}. 
 More precisely, it was shown that the multiple Eisenstein series can be written as an explicit $\mz[2\pi i]$-linear combination of the $q$-series
 \begin{align}\label{eq:defqserg}
   g(k_1,\dots,k_r) \df \sum_{\substack{m_1 > \dots > m_r > 0\\ n_1, \dots , n_r > 0}} \frac{n_1^{k_1-1}}{(k_1-1)!} \dots \frac{n_r^{k_r-1}}{(k_r-1)!}  q^{m_1 n_1 + \dots + m_r n_r }.
 \end{align}
 
We denote the $\Q$-vector space spanned by all $\eqref{eq:defmes}$ by 
 \begin{align}\label{eq:defames}
     \ames \df \langle \mathbb{G}_{k_1,\dots,k_r} \mid  r\geq 0, k_1,\dots,k_r \geq 2\rangle_\Q\,,
 \end{align}
where we set $\mathbb{G}_{k_1,\dots,k_r}(\tau)=1$ for $r=0$. With a similar argument as for multiple zeta values, one can see that $\ames$ is an algebra. More precisely, the subspace $\h^2 = \Q\langle z_k \mid k\geq 2\rangle$ of $\h^1$ is closed under the stuffle product~$\ast$, i.e., we get a $\Q$-subalgebra $(\h^2,\ast) \subset (\h^1,\ast)$. The space~$\ames$ can then be seen as the image of the $\Q$-linear map $\mathbb{G}: \h^2 \rightarrow \mathcal{O}(\Ha)$ defined on the generators by $z_{k_1}\cdots z_{k_r} \mapsto \mathbb{G}_{k_1,\dots,k_r}\mspace{1mu}$. This map is an algebra homomorphism (\cite{Ba3}) with respect to the stuffle product~$\ast$. 

The Fourier expansion~\eqref{eq:mesfourier} and the discussion above raise the natural question of whether there exist algebra homomorphisms $\mathbb{G}^\bullet: (\h^1,\bullet) \rightarrow \mathcal{O}(\Ha)$, for $\bullet \in \{\ast, \shuffle\}$, such that $\mathbb{G}^\bullet_{\mid \h^2} = \mathbb{G}$ and such that $\mathbb{G}^\bullet(k_1,\dots,k_r)$ has a Fourier expansion with constant term $\zeta^\bullet(k_1,\dots,k_r)$ for any $k_1,\dots,k_r \geq 1$. The case $\bullet = \shuffle$ was positively answered in~\cite{BT}, where the authors constructed  \emph{shuffle regularized multiple Eisenstein series} $\mathbb{G}^\shuffle$ by relating the Goncharov coproduct to the Fourier expansion of multiple Eisenstein series. Similarly, the case $\bullet = \ast$ was solved in~\cite{Ba} (see also~\cite{Ba3}), where \emph{stuffle regularized multiple Eisenstein series} $\mathbb{G}^\ast$ were constructed, using the deconcatination coproduct for the stuffle algebra. As described at the beginning, for multiple zeta values it is conjectured that all relations can be described by the extended double shuffle relations, which then lead to the notion of formal multiple zeta values, given by symbols modulo these relations. Similarly, one can ask the question of what relations are satisfied by, for example, the stuffle regularized multiple Eisenstein series and how their formal version should be defined. 
In the following, we will explain the origin and motivation of formal multiple Eisenstein series, based on various numerical calculations in recent years. 

In both the construction of~$\mathbb{G}^\shuffle$ and of~$\mathbb{G}^\ast$, generalizations of the $q$-series~$g$ in~\eqref{eq:defqserg} were used. 
More precisely, for $k_1,\dots,k_r\geq 1,d_1,\dots,d_r\geq 0$ define the $q$-series
\begin{align}\label{def:big}
    g\bi{k_1,\dots,k_r}{d_1,\dots,d_r} \df \sum_{\substack{m_1>\dots>m_r>0\\n_1,\dots,n_r>0}} \frac{n_1^{k_1-1} m_1^{d_1}}{(k_1-1)!} \cdots \frac{n_r^{k_r-1} m_r^{d_r}}{(k_r-1)!} q^{m_1 n_1 + \dots + m_r n_r}\,.
\end{align}
The stuffle and shuffle regularized multiple Eisenstein series are $\mathcal{Z}[2\pi i]$-linear combinations of these $q$-series.
The $q$-series~\eqref{def:big} were first defined in~\cite{Ba}, where their algebraic structure was described in detail. In particular, it was shown that their product can be expressed by a quasi-shuffle product quite similar to the stuffle product. For example, we have
\begin{align}\label{eq:bigstuffleexample}
    g\bi{2}{1} g\bi{3}{2} = g\bi{2,3}{1,2} + g\bi{3,2}{2,1} + g\bi{5}{3} - \frac{1}{12}g\bi{3}{3}.
\end{align}
If we call $k_1+\dots+k_r+d_1+\dots+d_r$ the \emph{weight} of~\eqref{def:big}, then this product can be seen as a double indexed version of the usual stuffle product (where we add componentwise to obtain the `stuffle'-terms), modulo the lower weight term $g\bi{3}{3}$.
The coefficient of $q^N$ in these $q$-series can be related to sums over certain partitions of $N$ (see \cite{B4, BI}). Using the conjugation of partitions then naturally leads to a family of relations among these $q$-series, which we call \emph{swap invariance}. For example, in the case $r=1$ this just means $g\bi{k}{d} = \frac{d!}{(k-1)!} g\bi{d+1}{k-1}$, which can also be obtained directly from the definition. 
Based on numerical calculations (\cite{Ba2}), the first author conjectured that all relations among the $q$-series~\eqref{def:big} are a consequence of this swap invariance and the quasi-shuffle product. Moreover, it was observed that, conjecturally, every series~\eqref{def:big} can be written as a $\Q$-linear combination of the single indexed version $g(k_1,\dots,k_r)$ in~\eqref{eq:defqserg}. Finally, numerical calculations further suggest that multiple Eisenstein series $\mathbb{G}_{k_1,\dots,k_r}$ satisfy the same $\Q$-linear relations as $g(k_1,\dots,k_r)$ modulo lower weight terms. The quasi-shuffle product for the double indexed $g$ modulo lower weight terms can be seen as a double indexed version of the stuffle product as we can see in~\eqref{eq:bigstuffleexample}.

All these observations then lead to the notion of formal multiple Eisenstein series, which should be understood as formal objects satisfying the same relations as multiple Eisenstein series. 

\subsection{Formal multiple Eisenstein series}\label{sec:formal}
Motivated by the observations mentioned above, the algebra of formal multiple Eisenstein series $\fmes$ is defined (Definition~\ref{def:fmes}) to be the $\Q$-vector space spanned by symbols $\gf$, which are swap invariant and whose product is given by (a double indexed version of) the stuffle product. More precisely, for $r\geq 1, k_1,\dots,k_r\geq 1,d_1,\dots,d_r \geq 0$ we introduce formal variables $\gf\bi{k_1,\dots,k_r}{d_1,\dots,d_r}$ and impose the relations (swap invariance)
\begin{align}\label{eq:swapinvarianceintro}
		\genG_r \bi{X_1,\dots,X_r}{Y_1,\dots,Y_r}  = \genG_r \bi{Y_1 + \dots + Y_r,\dots,Y_1+Y_2,Y_1}{X_r,X_{r-1}-X_r,\dots,X_1-X_2}\,\qquad (\text{for all } r\geq 1),
\end{align}
where $\genG_r$ denotes the generating series 
\begin{align*}
\genG_r \bi{X_1,\dots,X_r}{Y_1,\dots,Y_r} := \sum_{\substack{k_1,\dots,k_r\geq 1\\d_1,\dots,d_r \geq 0}} \gf\bi{k_1,\dots,k_r}{d_1,\dots,d_r} X_1^{k_1-1} \cdots X_r^{k_r-1} \frac{Y_1^{d_1}}{d_1!} \cdots \frac{Y_r^{d_r}}{d_r!}\,.
\end{align*}
We then define the \emph{algebra of formal multiple Eisenstein series}~$\fmes$ as the $\Q$-vector space spanned by these formal symbols $\gf$ equipped with the \emph{stuffle product} (Definition~\ref{def:stuffleproduct}), e.g.,
{\footnotesize
\begin{align*}
	\gf\bi{k_1}{d_1}  \gf\bi{k_2}{d_2} &= \gf\bi{k_1,k_2}{d_1,d_2}+ \gf\bi{k_2,k_1}{d_2,d_1}+	\gf\bi{k_1+k_2}{d_1+d_2}\,,\\
	\gf\bi{k_1}{d_1}  \gf\bi{k_2,k_3}{d_2,d_3} &= \gf\bi{k_1,k_2,k_3}{d_1,d_2,d_3}+ \gf\bi{k_2,k_1,k_3}{d_2,d_1,d_3}+ \gf\bi{k_2,k_3,k_1}{d_2,d_3,d_1}+	\gf\bi{k_1+k_2,k_3}{d_1+d_2,d_3}+ \gf\bi{k_1, k_2+k_3}{d_1,d_2+d_3}\,.	
\end{align*}
}For example, in the case $r=1$ they satisfy, similarly to $g$, the relation $\gf\bi{k}{d} = \frac{d!}{(k-1)!} \gf\bi{d+1}{k-1}$. The use of double indices is necessary to make sense of the swap invariance~\eqref{eq:swapinvarianceintro}. But often we will be interested in the case when $d_1=\dots=d_r=0$ and set
\begin{align}\label{eq:singleGintro}
    \gf(k_1,\dots,k_r) \df \gf \bi{k_1,\dots,k_r}{0,\dots,0}\,.
\end{align}
Then the discussion above suggests that the symbols $\gf(k_1,\dots,k_r)$ should satisfy the same relations as the stuffle regularized multiple Eisenstein series\footnote{It should be noted the stuffle regularized multiple Eisenstein series are not defined uniquely. So, more precisely, the claim is that there is a choice in the stuffle regularization of the multiple Eisenstein series so that they satisfy exactly the same relations as their formal counterparts.}. In fact, we will see that the $\gf(k)$ satisfy exactly the algebraic relations as the classical Eisenstein series (Theorem~\ref{thm:relpevevk}), which is quite surprising considering the simple family of relations that we impose on $\gf$. Even though it seems like the single indexed $\gf(k_1,\dots,k_r)$ in~\eqref{eq:singleGintro} just give a small portion of the whole space~$\fmes$, we conjecture (Conjecture~\ref{conj:BBvI}) it is already spanned by them. This justifies calling $\fmes$ the algebra of formal multiple Eisenstein series. 

By the Fourier expansion~\eqref{eq:mesfourier} of multiple Eisenstein series we see that projecting onto the constant term gives a surjective algebra homomorphism from the space of multiple Eisenstein series to multiple zeta values. This raises the natural question of whether there is a formal analogue of this projection onto the algebra of formal multiple zeta values $\fmz$. We give a positive answer to this question together with an explicit description of the kernel of this `formal projection onto the constant term' as follows.
\begin{theorem}\label{thm:main1} There exists a surjective algebra homomorphism 
\begin{align*}
    \pi: \fmes \rightarrow \fmz,
\end{align*}
with $\pi(\gf(k_1,\dots,k_r))=\zeta^f(k_1,\dots,k_r)$. The kernel of $\pi$ is the ideal generated by all formal multiple Eisenstein series which are \emph{not} of the form
\begin{align*} 
    \gf\bi{1,\dots,1,k_1,\dots,k_r}{d_1,\dots,d_s,0,\dots,0},
\end{align*}
for some $r,s\geq 0$ and $k_1,\dots,k_r\geq 1, d_1,\dots,d_s \geq 0$. 
\end{theorem}
In fact, we will define the algebra of formal multiple zeta values as the quotient by the above ideal (Definition~\ref{def:algebraformalmzv}) and then show that this algebra is isomorphic to the classical definition of formal multiple zeta values (Theorem~\ref{thm:formalmzvdsh}). Theorem~\ref{thm:main1}  offers a new perspective on the extended double shuffle relations. Namely, it states that these relations are equivalent to the combination of swap invariance, the stuffle product, and the relations derived from dividing out the mentioned ideal.

\subsection{Derivations and \texorpdfstring{$\sltwo$}{sl2}-algebras}
The second and main part of this work focuses on the derivations of the algebra~$\fmes$. This is motivated by the derivations for classical (quasi)modular forms, which we recall now.  For this, we first introduce another normalization of the Eisenstein series and define 
\begin{equation}\label{eq:Gk}
G_k \df (-2\pi i)^{-k} \mathbb{G}_k = -\frac{B_k}{2k!} + \frac{1}{(k-1)!}\sum_{m, n \geq 1} m^{k-1} \, q^{m n}, \qquad (B_k = k\text{th Bernoulli number}).
\end{equation}
As is well-known, for even $k$ the relations
\begin{alignat}{2} 
\begin{split}\label{eq:classicalgkrelations}
   \frac{k+1}{2}G_{k}  &\= \frac{1}{k-2}\, q \frac{d}{dq}G_{k-2} \,+ \!\!\sum_{\substack{k_1+k_2=k \\ k_1, k_2 \geq 2 \text{ even} }}\!\! G_{k_1} G_{k_2}\,, \qquad (k\geq 4)\\
        \frac{(k+1)(k-1)(k-6)}{12}\, G_k &\= \!\!  \sum_{\substack{k_1+k_2 = k\\k_1, k_2\geq 4 \text{ even}}}\!\! (k_1-1)(k_2-1)   \, G_{k_1} G_{k_2}\,, \qquad (k\geq 6)
\end{split}
\end{alignat}
suffice to write every Eisenstein series $G_k$ with $k\geq 4$ as a polynomial in $G_4$ and $G_6\mspace{1mu}$. Moreover, the relations~\eqref{eq:classicalgkrelations} imply that the space $\qmf = \Q[G_2,G_4,G_6]$ of quasimodular forms (with rational coefficients) is closed under the operator $q \frac{d}{dq}$. Even more, the algebra~$\qmf$ is an $\sltwo$-algebra (see~\cite{Zag08} or Section~\ref{sec:sl2algebras}). In general, an \emph{$\sltwo$-algebra} is an algebra together with three derivations $W,D,\delta$ acting on this algebra and satisfying the following commutator relations
\begin{align*}
    [W,D]=2D, \quad [W,\deer]=-2\deer,\quad [\deer,D]=W.
\end{align*}
In the case of quasimodular forms, these derivations are given by $D=q \frac{d}{dq}$ and the other two derivations are determined by $\delta G_2 = -\frac{1}{2}$, $\delta G_4 = \delta G_6 = 0$ and $W G_k = k G_k\mspace{1mu}$. As one of the main results of this work (Theorem~\ref{thm:fmesissl2algebra}), we will show that the algebra~$\fmes$ is an $\sltwo$-algebra which can be seen as a natural generalization of the $\sltwo$-algebra of quasimodular forms:

\begin{theorem} There exist derivations $W,D,\delta$ on $\fmes$ such that
\begin{enumerate}[{\upshape(i)}]
    \item $\fmes$ is an $\sltwo$-algebra\emph{;}
    \item the subalgebra $\fqmf = \Q[\gf(2),\gf(4),\gf(6)] \subset \fmes$ is isomorphic to $\qmf$ as an $\sltwo$-algebra.
\end{enumerate}
\end{theorem}

We refer to $\fqmf$ as the algebra of \emph{formal quasimodular forms}. The derivations $W,D,\delta$ will be given explicitly. For example, the derivations~$W$ and $D$ on~$\fmes$ are given as the $\Q$-linear maps defined on the generators by 
\begin{align}\label{eq:defD}
    \D \gf \bi{k_1,\dots,k_r}{d_1,\dots,d_r}   &\df \sum_{j=1}^r k_j \,    	\gf \bi{k_1,\dots,k_j+1,\dots,k_r}{d_1,\dots,d_j+1,\dots,d_r},\\
    \W \gf \bi{k_1,\dots,k_r}{d_1,\dots,d_r}   &\df (k_1+\dots+k_r+d_1+\dots+d_r) \bi{k_1,\dots,k_r}{d_1,\dots,d_r}.
\end{align}
On the single indexed $\gf$ we will give an algebraic interpretation of the action of $\D$ in terms of double shuffle relations (see~\eqref{eq:Donfil0}), which, for example, gives $\D \gf(1) = \gf(3) - \gf(2,1)$. Applying the projection $\pi$ to this, together with the fact that $\D \fmes \in \ker \pi$, yields $\fzeta(3)=\fzeta(2,1)$.

The derivation~$\delta$ will be much more involved (see its explicit formula in the proof of Proposition~\ref{prop:deer}) and its discovery is based on extensive numerical experiments done by the authors. To describe it we will introduce some general derivations for quasi-shuffle algebras in Section~\ref{sec:derivationsforqsh} and then write $\delta$ as a sum of five of these derivations. As explained in Remark~\ref{rem:uniquedwdelta}, it seems that all three derivations on~$\fmes$ are unique with the property of the above theorem. The $\sltwo$-algebra structure of $\fmes$ gives a natural definition of the space of \emph{formal modular forms} $\fmf \df \ker \delta_{\mid \fqmf}$ and a notion of Rankin--Cohen brackets. The projection $\pi$ in Theorem~\ref{thm:main1} then also naturally leads to the space of \emph{formal cusp forms} $\fs \df \ker \pi_{\mid \fmf}\mspace{1mu}$. We show in Section~\ref{sec:formalquasimodularforms} that these spaces are isomorphic to their classical counterparts~$\mf$ and~$\mathcal{S}$.

\subsection{Overview} We end the Introduction by giving an overview of some of the spaces considered in this work. Summarizing the discussion above, we have the following \emph{`formal picture'}:
\vspace{-0.5cm}
\begin{figure}[H]
	\centering
\input{Diagrams/formal_picture}\hspace{0.89cm}
\end{figure}
\noindent Here, $\fil{\lwt}{0}{\fmes}$ is the subalgebra of $\fmes$ spanned by the single indexed $G(k_1,\dots,k_r)$ in~\eqref{eq:singleGintro} and $\fames$ is the subalgebra of~$\fmes$ spanned by all the symbols with $k_1,\dots,k_r \geq 2$. The algebra $\fames$ can be seen as the formal version of the algebra~$\ames$ spanned by the multiple Eisenstein series~\eqref{eq:defmes}. We conjecture that this space is also an $\sltwo$-subalgebra of $\fmes$ isomorphic to $\ames$ (Conjecture~\ref{conj:espacesl2}).

Multiple Eisenstein series were the original motivation for the project, and they are the centerpiece of the \emph{`classical picture'}:

\vspace{-0.4cm}
\begin{figure}[H]
	\centering
	\input{Diagrams/classical_picture}
\end{figure}
\noindent Even though this classical picture motivated the formal one, not all counterparts of the formal picture are clear classically. Here, $\qmfpi = \Q[\mathbb{G}_2,\mathbb{G}_4,\mathbb{G}_6]$ (with $\mfpi$ and $\cusppi$ the subspace of $\qmfpi$ of modular and cusp forms respectively) and $\ames^{\rm{reg}}$ denotes a space of regularized multiple Eisenstein series (e.g., one of the regularizations mentioned at the beginning). This algebra~$\ames^{\rm{reg}}$ should give the correct counterpart for $\fil{\lwt}{0}{\fmes}$, but it is difficult to show that there exists a well-defined homomorphism as no counterpart for $\fmes$ is known. 

In~\cite{BB} the authors were able to give a `rational analogue' of the space~$\fmes$. For this, denote by $\mz_q \subset \Q\llbracket q\rrbracket$ the space spanned by the $q$-series~\eqref{def:big}, which appeared in the regularization of multiple Eisenstein series. Then it was shown in~\cite{BB} that there exists a surjective algebra homomorphism $G: \fmes \rightarrow \mz_q$ (see Theorem~\ref{thm:cmes}), which conjecturally is also injective. We obtain the \emph{`rational picture'}:
\vspace{-0.4cm}
\begin{figure}[H]
	\centering
	\input{Diagrams/rational_picture}
\end{figure}
\noindent Here $\mz^\circ_q$ (resp., $\ames_\Q$) denotes the image of $\fil{\lwt}{0}{\fmes}$ (resp. $\fames$) under $G$. Conjecturally, all these spaces are isomorphic to their formal analogues and, therefore, are $\sltwo$-algebras, where the derivation $D$ is given by $q \frac{d}{dq}$. The open problem is to show that the derivation $\delta$ in $\fmes$ also gives a well-defined map on these spaces. The elements in the space~$\mz_q$ can be viewed as \emph{$q$-analogues of multiple zeta values} (see \cite{BK, BI}), since the limit $q\rightarrow 1$ yields multiple zeta values. \\

Since the start of this project in late 2019, and their original definition at that time, formal multiple Eisenstein series were studied in the case $r=2$ in~\cite{BKM} and some of the results proven here were announced in~\cite{B4}. Later, in~\cite{BB} the authors introduced a realization of the algebra~$\fmes$ in the space~$\Q\llbracket q\rrbracket$ by constructing combinatorial multiple Eisenstein series, which give the rational picture mentioned above. In her thesis~\cite{Bu1} (see also \cite{Bu2, Bu3}), Burmester then reinterpreted the algebra~$\fmes$ by introducing a new algebraic setup, called the balanced setup, which gives a different description of the stuffle product and the swap invariance. She also gave a new independent proof of Theorem~\ref{thm:main1} using the balanced setup (see \cite[Theorem~7.4]{Bu3}). We recall her algebraic setup in Section~\ref{sec:balanced}.

\subsection*{Acknowledgements}\mbox{}\\
The authors thank Annika Burmester, Masanobu Kaneko, Ulf Kühn and Don Zagier for fruitful discussions. 
This project was partially supported by JSPS KAKENHI Grants 19K14499, 21K13771 and 23K03030. The first author would like to thank the University of Cologne for the support of a research stay in which the authors worked on this project.
The second author was supported by the SFB/TRR 191 “Symplectic Structure in Geometry, Algebra and Dynamics”, funded by the DFG (Projektnr. 281071066 TRR 191). Part of the work was carried out during two visits of the second author to Nagoya, and he thanks Nagoya University for its hospitality and support. He also thanks Utrecht University, where he worked when the authors initiated this work. %, the Max Planck Institute for Mathematics, where he carried out the biggest part of the work, and the University of Cologne, where he finished this work. 

\section{Formal multiple Eisenstein series}
Define the set $\A$, whose elements we call \emph{letters}, by
\begin{align*}
\A = \left\{ \ai{k}{d} \mid k \geq 1,\,d \geq 0 \right\}.
\end{align*}
We are interested in $\Q$-linear combinations of words in the letters of $\A$, i.e., in elements of  $\QA$. Here and in the following, we call the monic monomials in $\QA$ \emph{words}.  
For $k_1,\dots,k_r\geq 1$ and $d_1,\dots,d_r\geq 0$, we will use the following notation to write words in $\QA$:
\begin{align*}
    \ai{k_1,\dots,k_r}{d_1,\dots,d_r} :=   	\ai{k_1}{d_1}\dots \ai{k_1}{d_1}\,,
\end{align*}
where the product on the right is the usual non-commutative product in $\QA$. The space~$\h^1$ defined in the introduction can be viewed naturally as a subspace of $\QA$ via the inclusion 
\begin{align}\begin{split}
\label{eq:identh1inqa}
\h^1 &\longrightarrow \QA\\
    z_{k_1}\cdots z_{k_r} &\longmapsto     \ai{k_1,\dots,k_r}{0,\dots,0}.
    \end{split}
\end{align}
We will extend the stuffle product defined on $\h^1$ to $\QA$ in the following way.

\begin{definition}\label{def:stuffleproduct}
    Define the \emph{stuffle product}  $\ast$ on $\QA$ as the $\Q$-bilinear product, which satisfies $1 \ast w = w \ast 1 = w$ for any word $w\in \QA$ and
    \begin{align*}
        \ai{k_1}{d_1} w \ast \ai{k_2}{d_2} v = \ai{k_1}{d_1} \left(w \ast \ai{k_2}{d_2} v\right) + \ai{k_2}{d_2} \left( \ai{k_1}{d_1} w \ast v \right) + \ai{k_1+k_2}{d_1+d_2} (w \ast  v) 
    \end{align*}
    for any letters $\ai{k_1}{d_1},\ai{k_2}{d_2} \in \A$ and words $w, v \in \QA$. 
\end{definition}
This turns $\QA$ into a commutative $\Q$-algebra $(\QA,\ast)$, as shown in~\cite{H}. The algebra $(\h^1,\ast)$ can be viewed as a subalgebra of $(\QA,\ast)$ via~\eqref{eq:identh1inqa}.
Most of the time, we will work with the generating series of our objects. Let $\B_0(\QA) = \QA$ and for $r\geq 1$ set
\begin{align}\label{eq:bimould}\B_r=\B_r(\QA) :=\QA\llbracket X_1,Y_1,\dots,X_r,Y_r \rrbracket \quad \text{and} \quad 
\B=\B(\QA) :=\bigoplus_{r\geq 0} \B_r(\QA).\qquad \end{align}
For $r\geq 0$ consider the family of formal power series $\genA=(\genA_0,\genA_1,\ldots)\in\B$, where $\genA_0=1$ and for $r\geq 1$ one has
\begin{align*}
    \genA_r\bi{X_1,\dots,X_r}{Y_1,\dots,Y_r} := \sum_{\substack{k_1,\dots,k_r\geq 1\\d_1,\dots,d_r \geq 0}} 	\ai{k_1,\dots,k_r}{d_1,\dots,d_r} X_1^{k_1-1} \dots X_r^{k_r-1} \frac{Y_1^{d_1}}{d_1!} \dots \frac{Y_r^{d_r}}{d_r!}\,.
\end{align*}

\begin{definition}\label{def:swap} We define the \emph{swap} as the linear map $\sigma: \B \rightarrow \B$ 
given by $\sigma(f_r)=(\sigma f_r)$ with
\begin{align}\label{eq:swapdef}
    \sigma\left(f_r\bi{X_1,\dots,X_r}{Y_1,\dots,Y_r}  \right) \df f_r\bi{Y_1 + \dots + Y_r,\dots,Y_1+Y_2,Y_1}{X_r,X_{r-1}-X_r,\dots,X_1-X_2}\,.
\end{align} 
Note that the swap~$\sigma$ is an involution.
\end{definition}

\begin{definition}
\label{obs:bimould} By applying a map $\rho: \B\to \B$ to the power series~$\genA$ one obtains a map $f: \QA\to\QA$ by comparing coefficients, i.e., $\rho\ai{k_1,\dots,k_r}{d_1,\dots,d_r}$ is defined as the coefficient of $X_1^{k_1-1} \dots X_r^{k_r-1} \frac{Y_1^{d_1}}{d_1!} \dots \frac{Y_r^{d_r}}{d_r!}$ of $\rho \genA_r\mspace{1mu}$
\end{definition}
In particular, we can think of the swap $\sigma: \QA\to\QA$ as the linear map defined in the above sense.

\begin{remark}There are several other ways to interpret the \emph{swap}~$\sigma$.
\begin{enumerate}[(i)]
\item 
\renewcommand{\vec}{\underline}
Explicitly,
\begin{align}
\sigma\ \ai{k_1, \dots , k_r}{d_1,\dots,d_r} = 
\sum_{\substack{a_1,\ldots,a_r\geq 1 \\ |\vec{a}|=|\vec{d}|+r}}\sum_{\substack{b_1,\ldots,b_r\geq 0 \\ |\vec{b}|=|\vec{k}|-r}}
C^{\vec{a},\vec{k}}_{\vec{b},\vec{d}}\, \ai{a_1,\ldots,a_r}{b_1,\ldots, b_r},
\end{align}
where for $\vec{a},\vec{b},\vec{d},\vec{k}\in \Z^r$ the constant $C^{\vec{a},\vec{k}}_{\vec{b},\vec{d}}$ is given by
\[C^{\vec{a},\vec{k}}_{\vec{b},\vec{d}} =(-1)^{|\vec{b}|}\prod_{j=1}^r\textstyle\binom{s_j(\vec{d})-s^j(\vec{a})+j-1}{a_{r-j+1}-1}\binom{k_{r-j+1}-1}{s_j(\vec{b})-s^j(\vec{k})+j-1} \frac{(a_j-1)!}{(k_j-1)!}(-1)^{s_j(\vec{k})+s^j(\vec{b})+j}
\]
and where for $\vec{\ell}\in \Z^r$ and $j=1,\ldots,r$, we write
\[ s_j(\vec{\ell}) = \sum_{i=1}^j \ell_i\,, \qquad s^j(\vec{\ell}) = \sum_{i=r-j+2}^r \ell_i\,.\]
\item The concept of `swap' originates from the works of \'{E}calle (see e.g., \cite{E}), where it acts as an important operator in his theory of moulds. In our context, its relevance emerges when exploring the $q$-series~\eqref{def:big}, which was a key motivator for introducing the space~$\fmes$. Their coefficients can be described as sums over certain partitions, and this also leads to a similar interpretation for the generating series:
The change of variables in the definition of the swap~\eqref{eq:swapdef} can be easily described by thinking of a Young diagram with rectangle widths $X_j$ and heights $Y_j\mspace{1mu}$. The change of variables can then be read off by taking the conjugate of this Young diagram as visualized in the following picture: 

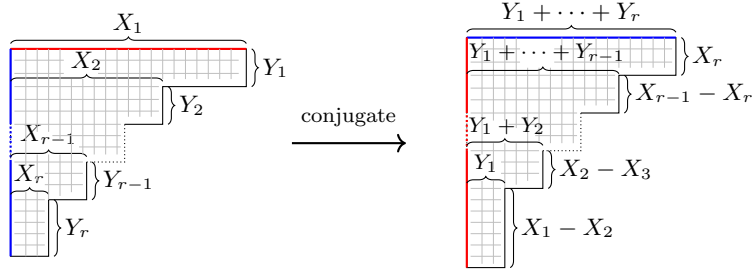
\vspace{-0.4cm}
\begin{figure}[H]
	\centering
	\input{Diagrams/partition}
\end{figure}
\noindent
For example, $X_1$ gets replaced by $Y_1+\dots+Y_r\mspace{1mu}$, since the new width of the top rectangle in the conjugated diagram comes from adding up all the heights from the original diagram.
\end{enumerate}

\end{remark}
Formal multiple Eisenstein series are objects whose generating series are invariant under the change of variables in~\eqref{eq:swapdef}. 
%In the sequel, we often make use of this construction. 

\begin{definition}\label{def:fmes}
We define the space of \emph{formal multiple Eisenstein series} as
\begin{align*}
    \fmes \df \faktor{ (\QA ,\ast)}{\I}\,,
\end{align*}
where $\I$ is the ideal in  $(\QA ,\ast)$ generated by $\sigma(w)-w$ for all $w\in \QA$. For $k_1,\dots,k_r\geq 1$ and $d_1,\dots,d_r\geq 0$ we denote the class of the word $w=\ai{k_1,\dots,k_r}{d_1,\dots,d_r}$ in $\fmes$ by $\gf(w)=\gf \bi{k_1,\dots,k_r}{d_1,\dots,d_r}$. We define a weight grading and two filtrations on $\fmes$, as follows:
\begin{enumerate}[(i)]\itemsep5pt
\item 
Write
\begin{align*}
\wt(w) & \df k_1+\ldots+k_r+d_1+\ldots+d_r, \\
\lwt(w) &\df  d_1+\ldots+d_r,  \\
\dep(w) &\df r
\end{align*}
for the \emph{weight},  \emph{lower weight} and \emph{depth} of the word~$w$ respectively. 
\item Write $\fmes_k$ for the subvectorspace of $\fmes$ generated by $\gf(w)$ with $\wt(w)=k$.
\item 
Write $\fil{\lwt}{d}\fmes$, $\fil{\dep}{r}\fmes$ for the lower weight and depth filtration on $\fmes$ respectively. We shorten the notation when considering two filtrations at the same time, that is
\begin{align}
\fil{\lwt}{d} \fmes &= \langle\, \gf(w) \mid w \in \A^\ast,\lwt(w) \leq d  \,\rangle_\Q\,,\\
\fil{\lwt,\dep}{d,r} \fmes &= \langle\, \gf(w) \mid w \in \A^\ast,\lwt(w) \leq d \text{ and } \dep(w) \leq r  \,\rangle_\Q\,, \text{ etc.}
\end{align}
\end{enumerate}
In the special case $\lwt(w)=0$, we write
\begin{align*}
    \gf(k_1,\dots,k_r) \df \gf \bi{k_1,\dots,k_r}{0,\dots,0}\,.
\end{align*}
\end{definition}

The space of formal multiple Eisenstein series is a  commutative $\Q$-algebra $(\fmes,\ast)$ where each element is swap invariant. Notice that the $\Q$-linear map
\begin{align}\begin{split}\label{eq:Gffromh1}
    \gf: (\h^1,\ast) &\longrightarrow \fmes\\
    z_{k_1}\cdots z_{k_r} &\longmapsto \gf(k_1,\dots,k_r)
    \end{split}
\end{align}
is an algebra homomorphism. Even though the double indices are crucial for the definition, a non-trivial conjecture is that the space~$\fmes$ is already spanned by the singles indexed $\gf(k_1,\dots,k_r)$, i.e., that the mapping~\eqref{eq:Gffromh1} is surjective. This conjecture is the formal version of the conjectures in \cite[Conjecture~4.3]{Ba}, \cite[Conjecture~5~(B2)]{BK} and~\cite[Conjecture~3.15]{BI}, of which only special cases are known (cf. \cite{Ba,Bu1,V}). 

\begin{conjecture}\label{conj:BBvI}
The map~\eqref{eq:Gffromh1} is surjective, i.e. 
\begin{align*}
    \fmes \simeq \fil{\lwt}{0}{\fmes}. 
\end{align*}
\end{conjecture}
A more refined version of the conjecture is that
\[\fil{\lwt,\dep}{d,r}{\fmes} \subset \fil{\lwt,\dep}{0,d+r}{\fmes} \qquad \text{for all } d,r\geq 0.\]

\begin{remark}\label{rem:anotherdefoffmes}
Conjecture~\ref{conj:BBvI} raises the natural question of whether we can give an explicit description of the kernel of~\eqref{eq:Gffromh1}, that is, if there exists an explicit description of an ideal $\mathfrak{R} \subset (\h^1,\ast)$, such that $\fmes \cong \faktor{(\h^1,\ast)}{\mathfrak{R} }$.
This would make it easier to compare the relations satisfied by (formal) multiple Eisenstein series with those satisfied by (formal) multiple zeta values. Conjecturally, the space of (stuffle regularized) multiple zeta values is isomorphic to $ \faktor{(\h^1,\ast)}{\mathfrak{E} }$, where $\mathfrak{E}$ is the ideal generated by $ w \ast v - w \shuffle v$ with $w \in \h^1, v \in \h^0$ (extended double shuffle relations). Since multiple zeta values should satisfy at least those relations of multiple Eisenstein series (as their constant terms), we would in particular expect $\mathfrak{R} \subset \mathfrak{E}$. 
\end{remark}

\begin{remark} Let $\fmes_k$ denote the space of formal multiple Eisenstein series of weight $k$. 
The conjectured dimensions for them are as follows:
\begin{align*}
\sum_{k\geq 0} \dim \fmes_k\, X^k &\overset{?}{=} \frac{1}{1-X-X^2-X^3+X^6+X^7+X^8+X^9}\\
&= 1+X+2 X^2+4 X^3+7 X^4+13 X^5+23 X^6+41 X^7+73 X^8+ \dots \,.
\end{align*}
This conjecture emerges from combining the dimension conjectures for $\mz_q\mspace{1mu}$ as discussed in~\cite{BK}, with the hypothesis that all relations in $\mz_q$ are derived from the swap invariance and the stuffle product. It should be noted that the dimensions given in~\cite[Remark~4.14]{BT} corroborate these figures. Multiplying the above series by $(1-X)$, we can extract the dimensions of the (admissible) multiple Eisenstein series presented there.
\end{remark}

\subsection{Relations among formal multiple Eisenstein series}\label{sec:rel}
We will now show some relations among formal multiple Eisenstein series in small depth. 
Recall the generating series of the formal multiple Eisenstein series is denoted by 
\begin{align*}
\genG_r \bi{X_1,\dots,X_r}{Y_1,\dots,Y_r} = \sum_{\substack{k_1,\dots,k_r\geq 1\\d_1,\dots,d_r \geq 0}} \gf\bi{k_1,\dots,k_r}{d_1,\dots,d_r} X_1^{k_1-1} \cdots X_r^{k_r-1} \frac{Y_1^{d_1}}{d_1!} \cdots \frac{Y_r^{d_r}}{d_r!}\,.
\end{align*}
Since the formal multiple Eisenstein series are swap invariant we have 
\begin{align}
\genG_r \bi{X_1,\dots,X_r}{Y_1,\dots,Y_r} = \genG_r \bi{Y_1 + \dots + Y_r,\dots,Y_1+Y_2,Y_1}{X_r,X_{r-1}-X_r,\dots,X_1-X_2}\,.
\end{align} 

On $\QA$ we can define another product~$\oshh$ by $w \osh v = \sigma( \sigma(w) \ast \sigma(v))$ for $w,v \in \QA$. One can easily check that since $\sigma$ is an involution, this product is commutative and associative. Due to the swap invariance of $\fmes$, this product is the same as the product~$\ast$. This implies a large family of relations among elements in $\fmes$, i.e., $f\ast g-f\osh g=0$ for all $f,g\in \fmes.$ These relations can be seen as an analogue of the double shuffle relations for multiple zeta values. If $f$ and $g$ are of depth~$1$, these relations are given as follows.
\begin{proposition} \label{prop:dshindep1} For $k_1,k_2\geq 1, d_1,d_2 \geq 0$ we have 
\begin{align*}
    \gf\bi{k_1}{d_1} \,\gf\bi{k_2}{d_2}=&\,\, \gf\bi{k_1,k_2}{d_1,d_2}+\gf\bi{k_2,k_1}{d_2,d_1}+\gf\bi{k_1+k_2}{d_1+d_2} \\
    = &\!\!\!\sum_{\substack{l_1+l_2=k_1+k_2\\ e_1+e_2=d_1+d_2}} \!\!\left(\!\binom{l_1-1}{k_1-1}\binom{d_1}{e_1}(-1)^{d_1-e_1} +   \binom{l_1-1}{k_2-1}\binom{d_2}{e_1} (-1)^{d_2-e_1} \!\right) \gf\bi{l_1,l_2}{e_1,e_2} \\&\quad +\frac{d_1! d_2!}{(d_1+d_2+1)!}\binom{k_1+k_2-2}{k_1-1}\gf\bi{k_1+k_2-1}{d_1+d_2+1}\,,
\end{align*}
where we sum over all $l_1,l_2 \geq 1$ and $e_1,e_2 \geq 0$, subject to $l_1+l_2=k_1+k_2$ and $e_1+e_2=d_1+d_2\mspace{1mu}$, in the second expression.
\end{proposition}
\begin{proof}
The first expression is a direct consequence of the definition of the product in $\fmes$. For the second, one calculates the coefficient of $X_1^{k_1-1} X_2^{k_2-1} \frac{Y_1^{d_1}}{d_1!} \frac{Y_1^{d_2}}{d_2!} $ in $\genG_1\bi{X_1}{Y_1} \osh 	\genG_1\bi{X_2}{Y_2}$, where
\begin{align*}
\genG_1\bi{X_1}{Y_1} \osh & \genG_1\bi{X_2}{Y_2}  
 \\& \= \genG_2\bi{X_1+X_2, X_2}{Y_1, Y_2-Y_1}+\genG_2\bi{X_1+X_2,X_1}{Y_2, Y_1-Y_2}+\frac{\genG_1\bi{X_1+X_2}{Y_1}-\genG_1\bi{X_1+X_2}{Y_2}}{Y_1-Y_2} \,.  \qedhere
    \end{align*}
\end{proof}
Proposition~\ref{prop:dshindep1} shows that the formal multiple Eisenstein series in depth two give a realization of the formal double Eisenstein space introduced in~\cite[Definition~2.1]{BKM}, since the latter are formal symbols satisfying the above relations. It was then shown in~\cite[Theorem~4.4]{BKM}, that these relations can be used to obtain the following relations.
\begin{theorem}\label{thm:relpevevk} For all $k_1,k_2\geq 1$ with $k=k_1+k_2\geq 4$ even we have 
\begin{align}
\begin{split}
    \frac{1}{2}\left(  \binom{k_1+k_2}{k_2} - (-1)^{k_1}\right) \gf(k) =  &\sum_{\substack{j=2\\j \text{even}}}^{k-2} \left( \binom{k-j-1}{k_1-1} + \binom{k-j-1}{k_2-1} - \ind_{j,k_1} \right)  \gf(j) \,  \gf(k-j)   \\
    &\, + \frac{1}{2} \left( \binom{k-3}{k_1-1} + \binom{k-3}{k_2-1}  + \ind_{k_1,1} + \ind_{k_2,1} \right) \gf\bi{k-1}{1}\,. 
\end{split}
\end{align}
\end{theorem}
\begin{proof}
This follows by sending $P\bi{j,k-j}{0,0}$ to $\gf(j)\gf(k-j)$ in~\cite[Theorem~4.4]{BKM}
\end{proof}

The following relations are special cases of Theorem~\ref{thm:relpevevk}, which will be used later when dealing with formal (quasi)modular forms. They can be seen as the formal version of the classical recursive formulas for Eisenstein series given in~\eqref{eq:classicalgkrelations}.
\begin{corollary} \label{cor:mfprod}
\begin{enumerate}[{\upshape (i)}]
\item For even $k\geq 4$ we have
\begin{align*}
    \frac{k+1}{2}\,\gf(k) \= \gf\bi{k-1}{1}  + \!\!\!\sum_{\substack{k_1+k_2=k \\ k_1, k_2\geq 2 \text{ even} }}\!\!\! \gf(k_1) \,\gf(k_2)\,.
\end{align*}
\item 	 For all even $k\geq 6$ we have
\begin{align*}
    \frac{(k+1)(k-1)(k-6)}{12}\,	\gf(k) \=  \!\!\!\sum_{\substack{k_1+k_2 = k\\k_1,k_2\geq 4 \text{ even}}} \!\!\! (k_1-1)(k_2-1)   \,\gf(k_1)   \,\gf(k_2)\,.
\end{align*}
\end{enumerate}

\end{corollary}
Due to Euler we know that for $m\geq 1$ we have 
\begin{align*}
    \zeta(2m)= - \frac{B_{2m}}{2 (2m)!} (-2\pi i)^{2m}= - \frac{B_{2m}}{2 (2m)!} (-24 \zeta(2))^{m}.
\end{align*}
As an analogue, in our formal setup we can show the following.
\begin{corollary}[Generalized Euler relation]\label{cor:eulerrelation} For $m\geq 1$  we have
		\begin{align*}
			\gf(2m) = - \frac{B_{2m}}{2 (2m)!} (-24 \gf(2))^{m} + Q_{2m},
		\end{align*}	
  for $Q_{2m} \in D\fqmf$, where $\fqmf=\Q[\gf(2),\gf(4),\gf(6)]$ is the space of formal quasimodular forms. 
\end{corollary}
\begin{proof}
First extend Corollary~\ref{cor:mfprod}(i) to all $k=2m\geq 2$ by using the swap relation $\gf\bi{1}{1}= \gf(2)$:
\begin{align} \label{eq:2masprod}
    \gf(2m) = \frac{2}{2m+1}\gf\bi{2m-1}{1}  + \frac{1}{3}\one_{m=1}\, \gf(2) + \frac{2}{2m+1}\sum_{\substack{a+b=m \\ a, b\geq 1 }} \gf(2a)\, \gf(2b)\,.
\end{align}
Now, define the generating series. 
\begin{align*}
    G(X) = \sum_{m\geq 1} \gf(2m) \,X^{2m},\qquad F(X) = \sum_{m\geq 2} \gf\bi{2m-1}{1}\, X^{2m} =  X^2\sum_{m\geq 1} \frac{\D\gf(2m)}{2m} X^{2m}\,.
\end{align*}
Then \eqref{eq:2masprod} is equivalent to the following differential equation
\begin{align}\label{eq:Gdiffeq}
    2 G(X)^{2} = G(X) + X G'(X) - 3 \gf(2)X^2 - 2 F(X)\,.
\end{align}

Now assume we have a formal power series of the form $B(X)= \sum_{m\geq 1} b(2m) X^{2m}$ which satisfies the differential equation 
\begin{align}\label{eq:Bdiffeq}
    2 B(X)^{2} = B(X) + X B'(X) - 3 b(2) X^2.
\end{align} Then the coefficients $b(2m)$ are all recursively determined by $b(2)$, i.e., for a fixed $b(2)$ the series~$B(X)$ satisfying~\eqref{eq:Bdiffeq} is unique. Now one checks by direct calculation that the formal power series 
\begin{align*}
B(X) = \sum_{m\geq 1} b(2m) X^{2m} \df& \frac{1}{2} - \frac{\sqrt{-24 b(2)} X}{4} \coth\left( \frac{\sqrt{-24 b(2)} X}{2}\right ) \\ =& - \sum_{m \geq 1} \frac{B_{2m}}{2 (2m)!} (-24 b(2))^m X^{2m} 
\end{align*}
is a solution of~\eqref{eq:Bdiffeq}, i.e., $b(m) = - \frac{B_{2m}}{2 (2m)!} (-24 b(2))^{m} $ for all $m\geq 1$. Since $G$ satisfies the differential equation~\eqref{eq:Gdiffeq}, we see that $\gf(2m)$ is also recursively determined by $\gf(2)$ and---due to the additional term $-2F(X)$---by $D\gf(2l)$ for $l<m$. The statement follows inductively since any $\gf(2l)$ is contained in $\fqmf$ by Corollary~\ref{cor:mfprod}.
\end{proof}

\subsection{Formal multiple zeta values} 
In this section, we define the space of formal multiple zeta values. Conjecturally, these satisfy exactly the same relations as multiple zeta values. The definition in this work is equivalent to the usual definition of formal multiple zeta values as formal symbols modulo the extended double shuffle relations. The difference is that here we define the space of formal multiple zeta values as a quotient of~$\fmes$. This approach has the benefit of allowing connections to $q$-analogues of multiple zeta values and modular form on a formal level, which is not directly possible with the usual approach. This construction is motivated by the work~\cite{BI}, where the authors consider the behaviour of the $q$-series~\eqref{def:big} as $q\rightarrow 1$. 
We define the following two subsets of the alphabet $\A$. 
\begin{align*}
    \A_0 \df \left\{ \ai{k}{0} \mid k \geq 1 \right\} \,,\qquad 	\A^1 \df \left\{ \ai{1}{d} \mid d \geq 0 \right\} \,.
\end{align*}
With this we define the following ideal in $(\QA,\ast)$ generated by the set $\A^* \backslash ((\A^1)^* (\A_0)^*)$
\begin{align}\label{eq:N}
    \mathfrak{N} \df \left( \A^* \backslash ((\A^1)^* (\A_0)^*)\right)_{\QA}\,,
\end{align}
where for an alphabet $\mathcal{L}$ by $\mathcal{L}^*$ we denote the set of words in the letters $\mathcal{L}$, i.e., the free monoid generated by the elements in $\mathcal{L}$.
The generators of $\mathfrak{N}$ are exactly those elements which are \underline{not} of the form
\begin{align} 
    \gf\bi{1,\dots,1,k_1,\dots,k_r}{d_1,\dots,d_s,0,\dots,0}\,,
\end{align}
for some $k_1,\dots,k_r\geq 1, d_1,\dots,d_s \geq 0$. 

\begin{definition}\label{def:algebraformalmzv}
    The algebra of \emph{formal multiple zeta values} is defined by 
    \begin{align*}
        \fmz \df \faktor{\fmes\,}{\mathfrak{N}}\,.
    \end{align*}
\end{definition}

The justification for the name formal multiple zeta values comes from the fact that our notion is equivalent, up to the non-vanishing of $\fzeta(1)$ in our case, to the one by Racinet~\cite{R} (see Theorem~\ref{thm:formalmzvdsh} below), which consists of formal symbols satisfying the extended double shuffle relations. In particular, we expect $\fmz \cong \mathcal{Z}[T]$. Note that this definition does not coincide with the definition of formal multiple zeta values in the Introduction: it is the content of Theorem~\ref{thm:formalmzvdsh} that both definitions are equivalent. 

We denote the canonical projection of the space of formal multiple Eisenstein series into the space of formal multiple zeta values by 
\begin{align}\label{eq:pi}
    \proj: 	\fmes \longrightarrow \fmz\,.
\end{align}
This projection can be seen as a formal version of the `projection onto the constant term'.

\begin{proposition}\label{prop:profsurj}
    The map $\proj_{|\fmesz}: \fmesz \rightarrow \fmz$ is surjective.
\end{proposition}
\begin{proof}
All non-zero elements in $\fmz$ are linear combinations of elements of the form 
\begin{align}\label{eq:shapeoff}
    f = \gf\Bigg(\ebi{1,\dots,1,}{d_1,\dots,d_s,}\underbrace{\ebi{1,\dots,1}{0,\dots,0}}_{j}\ebi{,k_1,\dots,k_r}{,0,\dots,0}\Bigg) 
\end{align}
with $s,r \geq 0$, $j\geq 0$, $d_s \geq 1$ and $k_1 \geq 2$. By induction on $j$ one can apply the usual calculation used for the stuffle regularization (\cite[Proposition~4.18]{BI}) to show that $f$ can be written as 
\begin{align*}
    f \equiv \sum_{m=0}^j f_m \, \gf(1)^{m} \quad \mod \mathfrak{N}\,,
\end{align*}
where the $f_m$ are have the same shape as~\eqref{eq:shapeoff} with $j=0$. Such elements can be expressed as products $\gf\bi{1,\dots,1}{d_1,\dots,d_s} \gf\bi{k_1,\dots,k_r}{0,\dots,0}$ modulo $\mathfrak{N}$. By the definition of $\sigma$ it is easy to see that $\sigma( \Q\langle \A^1\rangle) =  \Q\langle \A_0 \rangle$, i.e., $\sigma\gf\bi{1,\dots,1}{d_1,\dots,d_s} \in \fmesz$. Therefore, we can find a representative for the class of~$f$ in~$\fmz$ which is an element in~$\fmesz$.
\end{proof}

\begin{definition}\label{def:formalmzv}
For $k_1,\dots,k_r\geq 1$ we define the  \emph{formal multiple zeta value} $\fzeta(k_1,\dots,k_r)$  by
\begin{align*}
    \fzeta(k_1,\dots,k_r)\df \proj(\gf(k_1,\dots,k_r))\,.
\end{align*}
\end{definition}

By Proposition~\ref{prop:profsurj} the $\fzeta$ span the space~$\fmz$ of formal multiple zeta values. Recall that the derivation $D$ (see~\eqref{eq:defD} in the introduction) is given by
\begin{align}\label{eq:derivonz}
    \D \gf \bi{k_1,\dots,k_r}{d_1,\dots,d_r}   \df \sum_{j=1}^r k_j \,    	\gf \bi{k_1,\dots,k_j+1,\dots,k_r}{d_1,\dots,d_j+1,\dots,d_r}\,.
\end{align}
 Since the formal multiple zeta values can be thought of as the `constant term' of formal multiple Eisenstein series, they should vanish under taking the derivative as we will see now.
\begin{proposition}\label{prop:Dinkerphi} We have $\D \fmes \subset \ker(\proj)$.
\end{proposition}		
\begin{proof}
By~\eqref{eq:derivonz} we see that $\D \fmes \subset \mathfrak{N}$ since it is a linear combination of words which all contain a letter~$\ai{k}{d}$ with $k\geq 2$ and $d\geq 1$.
\end{proof}
This proposition implies that all the relations we proved so far for elements in $\fmes$ are also true in $\fmz$ after setting all terms involving derivatives, or, more generally, elements in $\mathfrak{N}$, to zero. 
For example, as a direct consequence of Proposition~\ref{prop:dshindep1} we obtain the double shuffle relations in depth two: For $k_1,k_2 \geq 1$ we have 
\begin{align*}
    \fzeta(k_1)\,  \fzeta(k_2) &= \fzeta(k_1, k_2) +  \fzeta(k_2, k_1) + \fzeta(k_1+k_2) \\
    &= \sum_{l_1+l_2 = k_1+k_2} \left( \binom{l_1-1}{k_1-1} + \binom{l_1-1}{k_2-1} \right) \fzeta(l_1, l_2) + \ind_{k_1+k_2,2}\, \fzeta(2)\,.
\end{align*}
In particular, we obtain the relation $\fzeta(3) = \fzeta(2,1)$ by taking $k_1=1, k_2 = 2$. And as a consequence of Corollary~\ref{cor:eulerrelation} we get Eulers relation, i.e., for $m\geq 1$ we have 
\begin{align*}
    \fzeta(2m) = -\frac{B_{2m}}{2 (2m)!} \left(-24 \fzeta(2) \right)^{m}\,.
\end{align*}
The $\Q$-linear map defined on the generators by 
\begin{align}\begin{split}\label{eq:fzetamap}
    \fzeta: \h^1 &\longrightarrow \fmz\\
    z_{k_1}\cdots z_{k_r} &\longmapsto \fzeta(k_1,\dots,k_r)
    \end{split}
\end{align}
is an algebra homomorphism with respect to the stuffle product~$\ast$. This follows from the fact that~\eqref{eq:Gffromh1} is an algebra homomorphism and from the definition of $\fzeta(k_1,\dots,k_r)$. The justification for calling $\fzeta$ formal multiple zeta values comes from the following theorem, stating that they conjecturally satisfy the same relations as ( $\ast$-regularized) multiple zeta values, namely the extended double shuffle relations (cf. \cite{IKZ,R}). Notice that, in contrast to~\cite{R}, we have $\fzeta(1)\neq 0$.
\begin{theorem}\label{thm:formalmzvdsh} The formal multiple zeta values $\fzeta$ satisfy exactly the extended double shuffle relations, i.e.\ the kernel of the map~\eqref{eq:fzetamap} is the ideal generated by $ w \ast v - w \shuffle v$ for $w \in \h^1$ and $v \in \h^0$.
\end{theorem}
\begin{proof}
In this proof, we make use of the notation and results in the appendix. First observe that the element $F \in \mathcal{B}(\fmz)^\times$ given in depth $r$ by 
\begin{align*}
\sum_{\substack{k_1,\dots,k_r\geq 1\\d_1,\dots,d_r \geq 0}} \pi\left( \gf\bi{k_1,\dots,k_r}{d_1,\dots,d_r} \right) X_1^{k_1-1} \cdots X_r^{k_r-1} \frac{Y_1^{d_1}}{d_1!} \cdots \frac{Y_r^{d_r}}{d_r!}
\end{align*}
satisfies, by the definition of the ideal~$\mathfrak{N}$, conditions~(i) and~(ii) of Proposition~\ref{prop:bijectionv1}. Moreover, since the formal multiple Eisenstein series satisfy the formula of the stuffle product, we also see $F \in \cG(\fmz)$. By Proposition~\ref{prop:appendix} we then find that the coefficients of the corresponding $H \in \cG_X(\fmz)$, which is exactly the generating series of the formal multiple zeta values $\fmz(k_1,\dots,k_r)$, satisfy exactly the extended double shuffle relations (as defined in Definition~\ref{def:regdsh}). This implies the statement in the theorem (cf. Remark~\ref{rem:racinetikz}).
\end{proof}

\begin{remark} We define the \emph{formal conjugated multiple zeta values} for $d_1,\dots,d_r \geq 0$ by
\begin{align*}
    \xi^f(d_1,\dots,d_r) \df \pi\left(\gf\bi{1,\dots,1}{d_1,\dots,d_r} \right).
\end{align*}
These span the space~$\fmz$ and satisfy the index shuffle product. They can be seen as the formal analogues of the conjugated multiple zeta values defined in~\cite[Definition~1.3]{BI}.
\end{remark}

\subsection{Realizations}
Let $A$ be a $\Q$-algebra. For a subalgebra $S \subseteq \fmes$ we will call an algebra homomorphism $\varphi: S \rightarrow A$ a \emph{realization of $S$ in $A$}. Notice that finding a realization of $\fmes$ in $A$ is equivalent to finding an algebra homomorphism 
\begin{align}
\begin{split}\label{eq:realization}
\varphi: (\QA, \ast) &\longrightarrow A\\
w = \ai{k_1,\dots,k_r}{d_1,\dots,d_r} &\longmapsto \varphi(w)
\end{split}
\end{align}
which is $\sigma$-invariant, i.e. $\varphi(\sigma(w)) = \varphi(w)$ for all words $w \in \QA$.
Since the definition of~$\fmes$ is motivated by multiple zeta values and multiple Eisenstein series, a natural question is whether there exist corresponding realizations of $\fmes$. More precisely, since multiple zeta values are just defined for indices $(k_1,\dots,k_r) \in \Z_{\geq 1}^r$ and multiple Eisenstein series are defined for indices $(k_1,\dots,k_r) \in \Z_{\geq 2}^r$ a natural question is whether there are $\sigma$-invariant algebra homomorphisms~\eqref{eq:realization} for $A=\mz$ (resp. $A=\mathcal{O}(\Ha)$) with $\ai{k_1,\dots,k_r}{0,\dots,0} \mapsto \zeta(k_1,\dots,k_r)$ for $k_1,\dots,k_r\geq 1$ (resp.  $\ai{k_1,\dots,k_r}{0,\dots,0} \mapsto \mathbb{G}_{k_1,\dots,k_r}$ for $k_1,\dots,k_r\geq 2$). In particular, one difficulty is that one has to extend these maps to all $\ai{k_1,\ldots,k_r}{d_1,\ldots,d_r}$. The case for $A=\mz$ was solved in~\cite[Theorem~4.23]{BI}. The case for $A=\mathcal{O}(\Ha)$ remains open as mentioned in the Introduction. Instead in~\cite{BB} the authors introduced a `rational version' of multiple Eisenstein series, which gives a realization in $A=\Q\llbracket q \rrbracket$ such that $\varphi(\ai{k}{0}) = (-2\pi i)^{-k} \mathbb{G}_k$ for $k\geq 2$ and $q=e^{2\pi i \tau}$. The construction of the images of words in higher depth is inspired by the formula for the Fourier expansion of the multiple Eisenstein series. 
Both constructions in~\cite{BI} and~\cite{BB} start with the family of $q$-series $g(\cdots)$ defined by the first author in~\cite{Ba}; see~\eqref{def:big}. 
Notice that in the special case $r=1$ and $k\geq 1$ one has $g\bi{k}{0}=\frac{1}{(k-1)!} \sum_{n>0} \sigma_{k-1}(n) q^n$. The key observation is that the linear map $\varphi_g: \ai{k_1,\dots,k_r}{d_1,\dots,d_r} \mapsto g\bi{k_1,\dots,k_r}{d_1,\dots,d_r}$ is $\sigma$-invariant (\cite{Ba}). This map is not an algebra homomorphism with respect to the product $\ast$, but one has $\varphi_g( w \ast v) \equiv \varphi_g(w) \varphi_g(v)$ modulo terms of lower weight. In~\cite{BB} the authors were able to correct these lower weight terms without losing the $\sigma$-invariance to obtain the following:

\newcommand{\cmes}{\varphi_G}
\newcommand{\zetaphi}{\varphi_\zeta}

\begin{theorem}[\cite{BB}]\label{thm:cmes} There exists an algebra homomorphism $\cmes: \fmes \rightarrow \Q\llbracket q \rrbracket$ such that
\begin{align*}
    \cmes: \gf(k) \longmapsto G(k) = -\frac{B_k}{2 k!} + \frac{1}{(k-1)!} \sum_{n>0} \sigma_{k-1}(n) q^n, \qquad (k\geq 2).
\end{align*}
Moreover, for any $f \in \fmes$ we have $\cmes( D(f) )  = q \frac{d}{dq} \cmes(f)$.
\end{theorem}
The image of $\cmes$ in Theorem~\ref{thm:cmes} for an arbitrary $\gf\bi{k_1,\dots,k_r}{d_1,\dots,d_r}$ is given by the combinatorial bi-multiple Eisenstein series value defined explicitly in~\cite[Definition~6.4]{BB}. Conjecturally, the space of bi-multiple Eisenstein series is isomorphic to $\fmes$. In particular, conjecturally, the realization~$\cmes$ in Theorem~\ref{thm:cmes} is injective.   

To obtain a realization of $\fmes$ in $\mz$ the authors in~\cite{BI} consider the (regularized) limit of $q\rightarrow 1$ of~\eqref{def:big} after multiplication with $(1-q)^{k_1+\dots+k_r+d_1+\dots+d_r}$. Lower weight terms also vanish under this limit and the resulting map is $\sigma$-invariant. As a consequence, one obtains the following:
\begin{theorem}[\cite{BI}]\label{thm:mzvrealization} There exists an algebra homomorphism $\cmes: \fmes \rightarrow \mz$ such that
\begin{align*}
    \zetaphi: \gf(k_1,\dots,k_r) \longmapsto \zeta^\ast(k_1,\dots,k_r), \qquad (k_1,\dots,k_r\geq 1).
\end{align*}
\end{theorem}
The image of $\varphi$ in Theorem~\ref{thm:mzvrealization} for an arbitrary $\gf\bi{k_1,\dots,k_r}{d_1,\dots,d_r}$ is given by the bi-multiple zeta value defined explicitly in~\cite[Definition~4.22]{BI}. Note that the statement of Theorem~\ref{thm:mzvrealization} is also a consequence of Theorem~\ref{thm:formalmzvdsh}: If $\pi_\mz: \fmz \rightarrow \mz$ denotes the projection of formal multiple zeta values to (stuffle regularized) multiple zeta values, then we get a realization $\pi_\mz \circ \pi: \fmes\to \mz$, where $\pi$ is given in~\eqref{eq:pi}. We should mention that the definition of the ideal~$\mathfrak{N}$ and our definition of $\fmz$ are motivated by the work~\cite{BI} for exactly this reason.

\section{Quasi-shuffle algebras and their derivations}\label{sec:derivationsforqsh}

\subsection{Quasi-shuffle products}
Suppose we have a set $\LL$, called the set of \emph{letters}, and a commutative and associative product $\diamond$ on $\LL$. Extending this product bi-linearly to $\Q \LL$ gives a commutative non-unital $\Q$-algebra $(\Q \LL,  \diamond)$. We will be interested in $\Q$-linear combinations of multiple letters of $\LL$, i.e., in elements of $\QL$. Here and in the following we call monic monomials in $\QL$ \emph{words}. Moreover, we call the degree of this monomial the \emph{length} of the word. 
In the special case $\LL=\A$, the following definition recovers Definition~\ref{def:stuffleproduct}.  
	
\begin{definition}
Define the \emph{quasi-shuffle product}  $\qsh$ on $\QL$ as the $\Q$-bilinear product which satisfies $1 \qsh w = w \qsh 1 = w$ for any word $w\in \QL$ and
    \begin{align}\label{eq:qshdef}
        a w \qsh b v = a (w \qsh b v) + b (a w \qsh v) + (a \diamond b) (w \qsh  v) 
    \end{align}
    for any letters $a,b \in \LL$ and words $w, v \in \QL$. 
\end{definition}
This gives a commutative $\Q$-algebra $(\QL,\qsh)$ as shown in~\cite{H}. In the special case that $a \diamond b=0$ for all $a,b \in A$, we obtain the usual shuffle product $\qsh=\shuffle$.

All quasi-shuffle algebras over the same alphabet are isomorphic due to the results in~\cite{H}. In particular, any quasi-shuffle algebra $(\QL,\qsh)$ is isomorphic to the shuffle algebra $(\QL,\shuffle)$ with the isomorphism given by 
\begin{align}\label{eq:deflog}
    \log_\diamond : (\QL,\qsh) &\longrightarrow (\QL,\shuffle)\\
    a_1 \cdots a_r &\longmapsto \\ &\!\!\!\sum_{\substack{1 \leq l \leq r\\ i_1 + \dots + i_l=r\\i_1,\dots,i_l \geq 1}} \frac{(-1)^{r-l}}{i_1 \cdots i_l} (a_1 \diamond \dots \diamond a_{i_1}) (a_{i_1+1}\diamond \dots \diamond a_{i_1+i_2}) \cdots (a_{i_1+\dots+i_{l-1}+1}\diamond \cdots \diamond a_{r})\,\,.
\end{align}
Its inverse is given by the algebra homomorphism
\begin{align*}
    \exp_\diamond : (\QL,\shuffle) &\longrightarrow (\QL,\qsh)\\
    a_1 \cdots a_r &\longmapsto \\ &\!\!\!\sum_{\substack{1 \leq l \leq r\\ i_1 + \dots + i_l=r\\i_1,\dots,i_l \geq 1}} \frac{1}{i_1! \cdots i_l!} (a_1 \diamond \dots \diamond a_{i_1}) (a_{i_1+1}\diamond \dots \diamond a_{i_1+i_2}) \cdots (a_{i_1+\dots+i_{l-1}+1}\diamond \cdots \diamond a_{r})\,.
\end{align*}
In particular, for any derivation $\Theta$ on $(\QL,\shuffle)$, the map
\[\Theta_\diamond \df \exp_\diamond \circ \,\Theta \circ \log_\diamond\]
is a derivation on $(\QL,\qsh)$.

\newcommand{\dletter}{\varphi}
\newcommand{\dword}{\Theta^\varphi}
\subsection{Derivations}
In this section, we consider examples of derivations on $(\QL,\qsh)$. Some of these derivations were studied previously in~\cite{KK} for the shuffle algebra, but it seems that not much research is done on derivations for quasi-shuffle algebras in general. We start with a simple construction, which gives a derivation on the level of words coming from derivations on the level of letters. Let $\dletter: \Q \LL \rightarrow \Q \LL$ be a $\Q$-linear map. We extend $\dletter$ to a $\Q$-linear map $\dword:\QL \rightarrow \QL$ by $\dword(1)=0$ and for a word $w=a_1\cdots a_r$ with $a_1,\dots,a_r \in \LL$ by
\begin{align}
   \dword(w) =  \dword(a_1\cdots a_r) = \sum_{j=1}^r a_1 \cdots \dletter(a_j) \cdots a_r\,.
\end{align}

\begin{proposition}\label{prop:derisder}
The map 
\begin{align}\label{eq:lettertoworddermap}
    \operatorname{Der}(\Q \LL,\diamond) &\longrightarrow  \operatorname{Der}(\QL,\qsh)\\
    \dletter &\longmapsto \dword
\end{align}
is a Lie algebra homomorphism. 
\end{proposition}
% \begin{proposition}\label{prop:derisder}
% If $\dletter$ is a derivation on $(\Q \LL,\diamond)$, then $\dword$ is a derivation on $(\QL,\qsh)$.
% \end{proposition}
\begin{proof}
\newcommand{\dw}{\Theta}
First, we show $\dword$ is a derivation. For this, we need to show $\dword(w \qsh v)  = \dword(w)\qsh v +  w \qsh \dword(v)$ for all words $v,w \in \QL$. From now on, we write $\dw$ for $\dword$. We proceed by induction on the sum of the length of $w$ and $v$. First of all, note the result clearly holds if $w$ or $v$ is the empty word. In particular, the smallest length case is clear.
For the general case we get by definition that for $a, b \in \LL$ and words $v,w \in \QL$
\begin{align*}
    \dw(a w& \qsh b v) \\
    \=&	\dw( a (w \qsh b v) + b (a w \qsh v) + (a \diamond b) (w \qsh  v) ) \\
    \=& 	\dletter(a) (w \qsh b v)  + a	\,\dw(w \qsh b v) +  \dletter(b) (a w \qsh v)  + b \,\dw(a w \qsh v) \,+ \\ &\qquad \dw(a \diamond b) (w \qsh  v) + (a \diamond b) \,\dw(w \qsh  v) \\
    \=& \dletter(a) (w \qsh b v)  + a	\,(\dw(w) \qsh b v)+ a\,(w \qsh \dw(b v))  +  \dletter(b) (a w \qsh v)  + b \,(\dw(a w) \qsh v )\, + \\
    &\qquad b \,(a w \qsh \dw(v))  + \dletter(a \diamond b) (w \qsh  v) + (a \diamond b) \,(\dw(w) \qsh  v + w \qsh \dw(v)),
\end{align*}
where we used the induction hypothesis. On the other hand, we have 
\begin{align*}
    \dw(a w) &\qsh b v + aw \qsh \dw(bv) \\
    \=&	\dletter(a) w \qsh b v  + a \dw(w)  \qsh b v   + a w \qsh \dletter(b) v   + a w \qsh b \,\dw(v)   \\
    \=& 	\dletter(a) (w \qsh b v) + b( \dletter(a) w \qsh v ) + (\dletter(a) \diamond b)(w \qsh v) + a (\dw(w)  \qsh b v) + b( a \dw(w)  \qsh v)\,+ \\ &+(a\diamond b) (\dw(w)  \qsh v)+  
  a (w \qsh \dw(b) v) +  \dletter(b)  ( a w \qsh v) + (a \diamond  \dletter(b) ) (w \qsh v) \,+ \\
  &+ a (w \qsh b \,\dw(v) )+ b(a w \qsh \dw(v) ) + (a \diamond b)( w \qsh \dw(v) )\,.
\end{align*}
Hence,
\[
\dw(a w \qsh b v) -\dw(a w) \qsh b v + aw \qsh \dw(bv) = \bigl(\dletter(a \diamond b)  - (\dletter(a) \diamond b)- (a \diamond  \dletter(b) )\bigr) (w \qsh v) ,
\]
which vanishes by using $\dletter$ is a derivation. 
Moreover, we have
\begin{align*}
[\Theta^{\dletter},\Theta^{\psi}](a_1\cdots a_r) \= &\sum_{i=1}^r\sum_{\substack{j=1\\ j\neq i}}^r a_1 \cdots \dletter(a_i)\cdots\psi(a_j)\cdots a_r  \\ & -\sum_{i=1}^r\sum_{\substack{j=1\\ j\neq i}}^r a_1 \cdots \dletter(a_j)\cdots\psi(a_i)\cdots a_r + \sum_{j=1}^r a_1 \cdots [\dletter,\psi](a_j)\cdots a_r \\
=\,& \Theta^{[\dletter,\psi]}(a_1\cdots a_r),% \qedhere
\end{align*}
which shows that \eqref{eq:lettertoworddermap} is a Lie algebra homomorphism.
\end{proof}

Later we will need the following generalization of Proposition~\ref{prop:derisder} for linear maps $\dletter$ which are not (necessarily) derivations with respect to $\diamond$, but satisfy the following identity
\begin{align}\label{eq:condphi}
2\dletter(a\diamond b\diamond c) = \dletter(a\diamond b)\diamond c + \dletter(b\diamond c)\diamond a + \dletter(c\diamond a)\diamond b,
\end{align}
 which is equivalent to the condition~\eqref{eq:gamma} on $\gamma$ below. Clearly, if $\dletter$ is a derivation on $(\Q \LL,\diamond)$ it satisfies~\eqref{eq:condphi}, but, since $(\Q \LL, \diamond)$ is a non-unital algebra, this identity does not imply that $\dletter$ is a derivation. Note that the map~$\dword$ below agrees with the map in the previous proposition in case $\dletter$ is a derivation. 
\begin{lemma}\label{lem:thetaphideriv}
Let $\dletter: \Q \LL \rightarrow \Q \LL$ be a $\Q$-linear map.
Let $\gamma:\Q\LL\times \Q\LL\to\Q\LL$ be the `failure of $\dletter$ being a derivation', i.e.,
\begin{align*}
    \gamma(a,b) \df \varphi(a \diamond b) - \varphi(a) \diamond b - a \diamond \varphi(b) \qquad (a,b\in \Q\LL).
\end{align*}
Assume $\gamma(a,b)=0$ for $a,b\in \im \diamond$ and
\begin{align}\label{eq:gamma}\gamma(a \diamond c,b)+\gamma(a,b\diamond c) = \gamma(a,b)\diamond c \qquad (a,b,c\in \Q\LL).
\end{align}
Then the $\Q$-linear map $\dword:\QL \rightarrow \QL$ defined by
\begin{align*}
 \dword(a_1\cdots a_r) &\df \sum_{j=1}^r a_1 \cdots \dletter(a_j) \cdots a_r
-\frac{1}{2}\sum_{j=1}^{r-1}  a_1 \cdots a_{j-1}\, \gamma(a_j,a_{j+1})\, a_{j+2} \cdots a_r 
\end{align*}
is a derivation on $(\QL, \qsh)$.
\end{lemma}
\begin{proof}
Again, we want to show $\dword(w \qsh v)  = \dword(w)\qsh v +  w \qsh \dword(v)$ for all words $w=w_1\cdots w_r$ and $v=v_1\cdots v_s\mspace{1mu}$. Setting $W=\{w_1,\dots,w_r\}$ and $V=\{v_1,\dots,v_s\}$ we have
\begin{align*}
    w \qsh v = \sum_{\substack{u=u_1 \cdots u_{t}\\ \max(r,s) \leq t \leq r+s}} \binom{u}{w,v}_{\!\!\diamond}\, u \,
\end{align*}
for some non-negative integers $\binom{u}{w,v}_{\!\!\diamond}$ and $u_j \in W \cup V\cup (W\diamond V)$. 

First, consider some $u_j\in W \cup V$  in the sum above, and write $u=u'u_ju''$. If $\dletter$ acts on $u_j\mspace{1mu}$, then the same contribution $u'\,\dletter(u_j)\,u''$ appears in both $\dword(w \qsh v)$ and $\dword(w)\qsh v+w\qsh \dword(v)$. Moreover, if both $u_j,u_{j+1}\in W$ (resp.~$V$), write $u=u'u_ju_{j+1}u''$. Then, if $\gamma$ acts on this pair, the same contribution \[-\frac{1}{2} 
u'\,\gamma(u_j,u_{j+1})\,u'' %- \frac{1}{2}u'\,\gamma(u_{j+1},u_{j})\,u''
\] appears in  both $\dword(w \qsh v)$ and $\dword(w)\qsh v$ (resp. $w\qsh \dword(v)$). 

Next, consider letters $w_m\in W$ and $v_n\in V$. This gives rise to words with letter $w_m\diamond v_n$ as well as to words with $w_m$ and $v_n$ as consecutive letters, i.e., $u = u' (w_m\diamond v_n) u''$, $u=u'w_m v_n u''$ or $u=u'v_nw_mu''$ for some words $u', u''$. 
Then, the term $\dword(w\qsh v)$
contains 
\[u' (\varphi(w_m \diamond v_n)) u''-\frac{1}{2} u' \gamma(w_m,v_n) u''-\frac{1}{2} u' \gamma(v_n,w_m) u''.\]
Moreover, the terms
$\dword(w)\qsh v+w\qsh \dword(v)$ contain  
\[u' (\varphi(w_m) \diamond v_n) u'' + u' (w_m \diamond \varphi(v_n)) u''\]
with the same multiplicity. Note that by definition of $\gamma$ these together vanish. 

The last case that we have to consider is $u=u'(w_m\diamond v_n) u_{j+1} u''$. First, assume $u_{j+1}\in V\cup W$. W.l.o.g., say $u_{j+1}=w_{m+1}\in W$. Then, the term $\dword(w\qsh v)$
contains 
\[-\frac{1}{2} u' (\gamma(w_m \diamond v_n,w_{m+1})) u''-\frac{1}{2} u' (\gamma(w_m ,w_{m+1}\diamond v_n)) u''\]
Moreover, the term
$\dword(w)\qsh v$ contains
\[-\frac{1}{2} u' (\gamma(w_m,w_{m+1}) \diamond v_n) u''.\]
Now, these contributions vanish by assumption~\eqref{eq:gamma} on $\gamma$.
Finally, assume $u_{j+1}\in V\diamond W$ with $u=u'(w_m\diamond v_n) u_{j+1} u''$. Then, the term $\dword(w\qsh v)$
contains 
\[-\frac{1}{2} u' (\gamma(w_m \diamond v_n,u_{j+1}) u''\]
Note that, as both $w_m\diamond v_n$ and $u_{j+1}$ are in the image of $\diamond$, this contribution is zero.  

Observe that we have covered all words appearing in $\dword(w \qsh v), \dword(w)\qsh v$ and $w \qsh \dword(v)$
and that each word appear with zero coefficient in 
\[\dword(w \qsh v)  - \dword(w)\qsh v -  w \qsh \dword(v),\]
i.e., $\dword$ is a derivation.
\end{proof}

\begin{proposition}\label{prop:daderivationshuffle} For any $a\in \LL$ the linear maps
\begin{align*}
   \dlefta:\QL &\longrightarrow \QL,&\quad    \drighta:\QL &\longrightarrow \QL\\
    a_1 \cdots a_r &\longmapsto \ind_{a_1=a} a_2 \cdots a_r,&\quad     a_1 \cdots a_r &\longmapsto \ind_{a_r=a} a_1 \cdots a_{r-1}, 
\end{align*}
are derivations on $(\QL,\shuffle)$.
\end{proposition}
\begin{proof}
Can be shown easily by induction on the length of a word.
\end{proof}

\begin{corollary}\label{cor:removeletterderqsh}
For any $a\in \LL$ the linear maps
\begin{align*}
   \dlefta_\diamond:\QL, &\longrightarrow \QL,\\
    a_1 \cdots a_r &\longmapsto \sum_{1\leq l\leq r} \ind_{a_1 \diamond \cdots \diamond a_l=a} \frac{(-1)^{l+1}}{l} a_{l+1} \cdots a_r, 
\end{align*}
\begin{align*}
   \drighta_\diamond:\QL, &\longrightarrow \QL,\\
    a_1 \cdots a_r &\longmapsto \sum_{1\leq l\leq r} \ind_{a_{r-l+1} \diamond \cdots \diamond a_r=a} \frac{(-1)^{l+1}}{l} a_{1} \cdots a_{r-l}, 
\end{align*}
are derivations with respect to the quasi-shuffle product $\qsh$.
\end{corollary}
\begin{proof}
This is a direct consequence of Proposition~\ref{prop:daderivationshuffle} by applying the isomorphisms $\log_\diamond : (\QL,\qsh) \to (\QL,\shuffle)$ and its inverse $\exp_\diamond$.
\end{proof}

\begin{proposition}\label{prop:phideriv} Let $a \in \LL$ and let $\varphi : \Q\LL \rightarrow \Q\LL$ be a linear map. Then the linear map 
\begin{align*}
    \dphia:\QL &\longrightarrow \QL,\\
    a_1 \cdots a_r &\longmapsto \sum_{j=1}^{r-1} \ind_{a_j=a}\, a_1\cdots a_{j-1} \varphi(a_{j+1}) \cdots a_r - \sum_{j=2}^r \ind_{a_j=a}\, a_1 \cdots \varphi(a_{j-1}) a_{j+1} \cdots a_r 
\end{align*}
is a derivation with respect to the shuffle product $\shuffle$.
\end{proposition}
\begin{proof}
Similar as in the proof of Lemma~\ref{lem:thetaphideriv}, for words $w,v \in \QL$ write
\begin{align*}
    w \shuffle v = w_1 \cdots w_r \shuffle v_1 \cdots v_s = \sum_{u=u_1 \cdots u_{r+s}} \binom{u}{w,v}\, u \,
\end{align*}
for some non-negative integers $\binom{u}{w,v}$.
Setting $W=\{w_1,\dots,w_r\}$ and $V=\{v_1,\dots,v_s\}$ we have $u_j \in W \cup V$. Assume we have a word $u$ in the sum on the right with $u_j=a$ and $u_j \in W$. If $u_{j+1} \in V$, then there also exists a word $u'$ in the shuffle product
which equals $u$ but with $u_{j+1}$ and $u_{j}$ interchanged and $\binom{u}{w,v} = \binom{u'}{w,v}$. In this case the terms in $\dphia$ cancel out. Therefore, the only contribution to $\dphia(w \shuffle v)$ with $u_j=a$ comes from those $u$, where $u_j,u_{j+1} \in W$ or $u_j,u_{j+1} \in V$ (similarly $u_j,u_{j-1} \in W$ or $u_j,u_{j-1} \in V$). But these terms come exactly from the terms in $\dphia(w) \shuffle v$ and $w \shuffle \dphia(v)$ respectively.
\end{proof}

Later we will need the following generalization of Proposition~\ref{prop:phideriv}. 
\begin{lemma}
 \label{lem:dphiaonqsh} Let $(\QL,\qsh)$ be a quasi-shuffle algebra and $S\subset \LL$ such that $\im(\diamond) \subset \Q \LL \backslash S$. For each $a\in S$ let $\varphi_a:\Q\LL\to\Q\LL$ be a linear map such that  $\varphi_a(b \diamond c)=\varphi_a(b) \diamond c$ for any $b,c \in \LL$ with $b\not\in S$ and $\varphi_a(a'\diamond c) = \varphi_{a'}(a\diamond c)$ for all $a,a'\in S$ and $c\in \LL$. Then the map $\Theta_S=\sum_{a\in S}\Theta^{\varphi_a,a}$ is a derivation on $(\QL,\qsh)$.
\end{lemma}

\begin{proof}
Notice that if $S=\{a\}$ and $\varphi(b \diamond c)=\varphi(b) \diamond c$ for all $b,c \in \LL$ then the statement follows immediately from Proposition~\ref{prop:phideriv} since $\dphia = \exp_\diamond \circ\, \dphia \circ \log_\diamond$. 
For general $S$, we proceed as in the proof of Lemma~\ref{lem:thetaphideriv}: We have
\begin{align*}
    w \qsh v = w_1 \cdots w_r \qsh v_1 \cdots v_s =  \sum_{\substack{u=u_1 \cdots u_{t}\\ \max(r,s) \leq t \leq r+s}} \binom{u}{w,v}_{\!\!\diamond}\, u \,
\end{align*}
for some non-negative integers $\binom{u}{w,v}_{\!\!\diamond}$ and $u_j \in W \cup V\cup(W\diamond V)$. Suppose $u_j=a$ for some $a\in S$. Moreover, without loss of generality, assume $u_j=a\in W$. We distinguish two possibilities for~$u_{j+1}\mspace{1mu}$. First, notice that we can treat the case $u_{j+1} \in V \cup W$ in a similar way as in the proof of Proposition \ref{prop:phideriv}. Next, we suppose $u_{j+1} = w_m \diamond v_n\mspace{1mu}$ and again distinguish two cases.  First, if $w_m\not \in S$, the corresponding term in $\dphia(w \qsh v)$ also appears in $\dphia(w) \qsh v$ with the same multiplicity as $\varphi_a(w_m\diamond v_n) = \varphi_a(w_m)\diamond v_n\mspace{1mu}$. Next, if $w_m=a'$ for some $a'\in S$, then $w$ contains two elements of $S$ consecutively, i.e., $u_j=w_{m-1}=a$ and $w_m=a'$ for some $a,a'\in S$. Therefore, the same term with opposite sign (and the same multiplicity) occurs for $u_{j+1}=a'$ and $u_j=a\diamond v_n\mspace{1mu}$, since $\varphi_a(a'\diamond v_n) = \varphi_{a'}(a \diamond v_n)$. Similar contributions in $\Theta_S(w) \qsh v$ cancel by the same argument. 
\end{proof}

\section{Derivations on \texorpdfstring{$\fmes$}{FMES}}
In this section, we consider derivations on the space~$\fmes$. All the derivations we consider come from derivations on $(\QA,\ast)$. 
Recall that we defined the alphabet $\A$ by
\begin{align*}
\A = \left\{ \ai{k}{d} \mid k \geq 1,\,d \geq 0 \right\} \,.
\end{align*}
The stuffle product $\ast$ defined in Definition~\ref{def:stuffleproduct} is then the quasi-shuffle product $\qsh$ associated to the product $\diamond$ on $\Q \A$ defined for $k_1, k_2 \geq 1$ and $d_1,d_2 \geq 0$ by 
\begin{align*}
    \ai{k_1}{d_1} \diamond \ai{k_2}{d_2} = 	\ai{k_1+k_2}{d_1+d_2}\,.
\end{align*}
We often use the results for derivations on quasi-shuffle algebras from Section~\ref{sec:derivationsforqsh}. 

Note that a derivation~$\Theta$ on $(\QA,\ast)$ restricts to $\fmes$ if and only if $[\Theta,\sigma] \QA \subset \I $. Namely, a derivation~$\Theta$ restricts to $\fmes$ if and only $\Theta\I\subset \I$. An element of $\mathcal{I}$ can be written as $(\sigma(w)-w)u$ for some words $w,u$. Now, we have
\[ \Theta((\sigma(w)-w)u) = ([\Theta,\sigma]w+\sigma(dw)-dw)u+(\sigma(w)-w)du \equiv ([\Theta,\sigma]w)u \mod \mathcal{I},\]
from which we conclude that $\Theta\I \subset \I$ if and only if $[\Theta,\sigma] \QA \subset \I $. In what follows we consider \emph{$\sigma$-equivariant derivations}, i.e., $[\Theta,\sigma]=0$.  In particular, every $\sigma$-equivariant derivation restricts to a derivation on $\fmes$.

\subsection{A derivation of weight \texorpdfstring{$-1$}{-1}}
Recall that any endomorphism of $\B$ defines an endomorphism on $\QA$ (see Definition~\ref{obs:bimould}). Therefore, for all $j=1,\ldots,r$ we define the following maps on $r$-tuples
\begin{align}
 p_j(X_1,\ldots,X_r) &= (X_1,\ldots,X_{j-1},X_{j+1},\ldots,X_r) \\
c_{j,j+1}(Y_1,\ldots,Y_r) &= (Y_1,\ldots,Y_{j-1},Y_j+Y_{j+1},Y_{j+2},\ldots,Y_r),
\end{align}
where we set $Y_{r+1}=0$. Note that $p_j$ removes the $j$th element and $c_{j,j+1}$ contracts the $j$th and $j+1$th element. As a special case, $c_{r,r+1}(Y_1,\ldots,Y_r)=(Y_1,\ldots,Y_{r-1})$. We abbreviate $\X=(X_1,\ldots,X_r)$ and similarly we define $\Y$. 
Then, define $\varphi_j^{+},\varphi_j^{-}:\B\to\B$ by
\begin{align} \varphi_j^{+} f_r(\X,\Y) \df 
\begin{cases}
f_{r-1}(p_j(\X),c_{j,j+1}(\Y)) & 1\leq j\leq r \\
0 & \text{else.}
\end{cases}\\
 \varphi_j^{-} f_r(\X,\Y) \df 
%f_{r-1}(p_j(\X),c_{j-1,j}(\Y))
\begin{cases}
f_{r-1}(p_j(\X),c_{j-1,j}(\Y)) & 2\leq j\leq r \\
0 & \text{else.}
\end{cases}
\end{align}
and define $\omega:\B\to \B$ by
\[ \omega \df  \sum_{j\geq 1}(\varphi_j^+-\varphi_j^-). \]
This induces an endomorphism $\omega : \QA \rightarrow \QA$.

\begin{proposition}
The map $\omega$ is a $\sigma$-equivariant derivation on $(\QA,\ast)$.
\end{proposition}
\begin{proof}

For the equivariance, we compute
\begin{align*}
 \varphi_j^{+} \sigma f_r(\X,\Y) &= \varphi_j^{+}f_r(Y_1 + \dots + Y_r,\dots,Y_1+Y_2,Y_1,X_r,X_{r-1}-X_r,\dots,X_1-X_2) \\
&= \sigma f_{r-1}(X_1,\ldots,X_{r-j},X_{r-j+2},\ldots,X_r,Y_1,\ldots,Y_{r-j},Y_{r-j+1}+Y_{r-j+2},\ldots,Y_r) \\
&= \sigma \varphi_{r-j+1}^{+} f_r(\X,\Y).
\end{align*}
Hence, 
\begin{align}\label{eq:phiandsigma}\varphi_j^{+} \,\sigma = \sigma\,\varphi_{r-j+1}^+ \qquad \text{and similarly} \qquad \varphi_j^{-} \,\sigma = \sigma\,\varphi_{r-j+2}^- \,.
\end{align}
These identities directly imply the equivariance.

For the second part, note that $\omega$ acts on words by
\begin{align*}
\omega \ai{k_1,\ldots,k_r}{d_1,\ldots,d_r} &= \ind_{(k_r,d_r)=(1,0)}\ai{k_1,\ldots,k_{r-1}}{d_1,\ldots,d_{r-1}} + 
\sum_{j=1}^{r-1} \ind_{k_j=1} \ai{k_1,\ldots,k_{j-1},k_{j+1},\ldots,k_r}{d_1,\ldots,d_{j-1},d_j+d_{j+1},\ldots,d_r} \\
&\qquad - 
\sum_{j=2}^r \ind_{k_j=1} \ai{k_1,\ldots,k_{j-1},k_{j+1},\ldots,k_r}{d_1,\ldots,d_{j-1}+d_j,d_{j+1},\ldots,d_r}.
\end{align*}
Hence, $\omega$ is a sum of the derivation~$\drighta_\diamond$ in Corollary~\ref{cor:removeletterderqsh} with $a=\ai{1}{0}$ and the derivation~$\dphia$
in Lemma~\ref{lem:dphiaonqsh} with $a=\ai{1}{d}$ and $\varphi\ai{k'}{d'}=\ai{k'}{d+d'}$ for all $d\geq 0$.
\end{proof}

Note that in the special case  $d_1=\dots=d_r$ the derivation $\omega$ is given by 
\[
\omega \ai{k_1,\ldots,k_r}{0,\ldots,0} =
\ind_{k_1=1} \ai{k_2,\ldots,k_r}{0,\ldots,0}.
\]
In particular, $\omega$ is the zero map if the sequence $(k_1,\ldots,k_r)$ is admissible. In the next section, we will see that the derivation $\omega$ allows to define the concept of regularization to all elements in $\fmes$.

\subsection{Derivations and polynomial representations}

% \begin{lemma}
% \label{lem:fundamentallemma}
% Let $A=\bigoplus_{n\geq 0} A_n$ be a graded $\Q$-algebra and $d\in \mathrm{End}(A)$  such that $d(A_n)\subset A_{n-k}$ for some $k>0$ (with $A_{-n}\df \{0\}$ for $n>0$). Assume $a\in A$ is such that $d(a)=1$. Then, $d$ is a derivation if and only if for any $f\in A$ one has 
% \begin{align}\label{eq:pol}
%  f= \sum_{i=0}^n f_i\, a^{i} \qquad \text{for $f_1,\ldots,f_n\in A$ satisfying $d(f_i)=0$}.\end{align}
% \end{lemma}
\begin{lemma}
\label{lem:fundamentallemma}
Let $A=\bigoplus_{n\geq 0} A_n$ be a graded $\Q$-algebra. Suppose $d$ is a derivation on $A$ such that for some $k>0$ we have $d(A_n)\subset A_{n-k}$ (with $A_{-n}\df \{0\}$ for $n>0$). Assume that there exists an $a\in A$ with $d(a)=1$. Then for all $f \in A$ there exist unique $f_1,\ldots,f_n\in \ker(d)$ such that 
\begin{align}\label{eq:pol}
 f= \sum_{i=0}^n f_i\, a^{i}.\end{align}
 Conversely, assume that for an $a \in A$ and a subalgebra $B\subset A$ there exist for any $f \in A$  unique $f_1,\ldots,f_n\in B$ such that \eqref{eq:pol} holds. Then the map $d \in \operatorname{End}(A)$ defined by 
\begin{align*}
    d(f) = \sum_{i=1}^n i f_i\, a^{i-1}
\end{align*}
is a derivation on $A$ with $d(a)=1$ and $ \ker(d) = B$.
\end{lemma}
\begin{proof}
Let $f\in A$ be given and assume that $d$ is a derivation. We prove the existence of an expansion of the form~\eqref{eq:pol} by induction on the \emph{depth} $p(f)=p=\min\{i\geq 0\mid d^{i+1}(f)=0\}$. Observe that since $d$ a derivation of weight $-k$ such a $p$ always exists. Also, observe that if $p=0$ we can write $f=f_0\mspace{1mu}$. 

Now, suppose the result is proven for elements of depth at most $p-1$ and $f$ is of depth $p$. Then, since $d$ is a derivation,
\[ d^p\Bigl(f-\frac{d(f)}{p}\,a\Bigr) = d^p(f) - \frac{d^p(f)}{p}\,p = 0.\]
Hence, both $f-\frac{d(f)}{p}\,a$ and $d(f)$ are of depth at most $p-1$, and the expansion of the form~\eqref{eq:pol} follows inductively. If $f=\sum_{i=0}^n f_i\, a^i =0$ we obtain $f_i=0$ by considering $d^j(f)$ for $j=n,\dots,0$, from which the uniqueness of the representation \eqref{eq:pol} follows.

For the converse, first of all, observe that by the uniqueness of the representation \eqref{eq:pol} the map $d$ is well-defined. Next, let $f,g\in A$ be given. Assume that both can be written as a polynomial in $a$ with coefficients in $B$, explicitly, 
\[f= \sum_{i} f_i \, a^{i}, \qquad g= \sum_{j} g_j \, a^{ j}\]
with $d(f_i)=0$ and $d(g_i)=0$ for all $i$. The converse statement then follows from the following explicit computation%Then, it follows from the following explicit computation that $d$ is a derivation:
\begin{align*} d(f\, g) &= d\Bigl(\sum_{i,j} f_i \, g_j \, a^{ i+j}\Bigr) \\
&= \sum_{i,j} (i+j)\, f_i \, g_j \, a^{i+j-1} \\
&= \sum_{i,j} i\, f_i \, g_j \, a^{ i-1}\, a^{ j} + \sum_{i,j} j\, f_i \, g_j \, a^{ i} \, a^{j-1} \\
&= d(f)\, g + f\, d(g). &&\qedhere
\end{align*}
Now, if $f\in B$, then by definition of $d$ we have $f=f_0$ and $d(f)=0$. Conversely, if $f\not \in B$, then $f_i\neq 0$ for some $i>0$ and $d(f)\neq 0$. 
\end{proof}

\begin{example}
Recall $\h=\Q\langle x,y\rangle$ and $\h^1=\Q\oplus \h y$. Recall that $\h^1$ is an algebra with respect to both the shuffle and stuffle product. Now, the map $\delta_y:\h^1\to \h^1$ given by
\[ 
\delta_y(w) = \begin{cases}
    v & \text{if } w=yv \text{ for some word } v\\
    0 & \text{else.}
\end{cases}
\]
is a derivation with respect to both products. Note that $\delta_y(w)=0$ precisely if $w$ is admissible. 
Hence, by Lemma~\ref{lem:fundamentallemma}, it follows that every $w\in \h^1$ can be written as a polynomial in $y$ with admissible coefficients. 
\end{example}

Similarly, on $\fmes$ we obtain:
\begin{corollary}
Every element $f\in\fmes$ can be written as
\[
f = \sum_{i} f_i \, \gf(1)^i
\]
for $f_i \in \fmes$ with $\omega(f_i)=0$. 
\end{corollary}
\begin{remark}
It is natural to say that an index $\ai{k_1,\ldots,k_r}{d_1,\ldots,d_r}$ is \emph{admissible} precisely if $\omega \gf\ai{k_1,\ldots,k_r}{d_1,\ldots,d_r}=0$. In the case $d_1=\ldots=d_r=0$ this notion of admissibility coincides with the usual one, i.e. $k_1\geq 2$. 
\end{remark}

\subsection{\texorpdfstring{$\sltwo$}{sl2}-algebras}\label{sec:sl2algebras}

An algebra $A$ is called an \emph{$\sltwo$-algebra} if %it forms a module over the Lie algebra~$\sltwo$ for which the elements of $\sltwo$ act by derivations 
there exists a Lie algebra homomorphism $\sltwo \rightarrow \operatorname{Der}(A)$. More explicitly, there exist three derivations $D,W,\deer \in \operatorname{Der}(A)$ such that $(D,W,\deer)$ forms an \emph{$\sltwo$-triple}. This means that the satisfy the commutator relations
\begin{align}\label{eq:commutationsrelations}
    [W,D]=2D, \quad [W,\deer]=-2\deer,\quad [\deer,D]=W\,.
\end{align}

As the main example of an $\sltwo$-algebra, we now discuss the algebra of quasimodular forms. Let~$\HH$ be the \emph{upper half plane}, i.e., the set of complex numbers~$\tau$ for which \mbox{$\Im(\tau)>0$}. The group~$\Gamma:=\sltwoz$ acts on~$\HH$ by \emph{M\"obius transformations}. Given $\gamma=\left(\begin{smallmatrix} a & b \\ c & d\end{smallmatrix}\right)\in\Gamma=\sltwoz$, $\phi:\HH\to \C$, $\tau\in\HH$ and $k\in \Z$, we let $\gamma\tau=\frac{a\tau+b}{c\tau+d}$ and define the \emph{slash operator} in weight~$k$ by
\[ (\phi|_k \gamma)(\tau) := (c\tau+d)^{-k} \phi(\gamma\tau).\]
A modular form of weight~$k$ is a holomorphic function $\phi:\HH\to\C$ such that $\phi|_k\gamma=\phi$ for all $\gamma\in \Gamma$ and such that $\phi$ admits a holomorphic $q$-expansion around infinity $(q=e^{2\pi \mathrm{i} \tau})$, i.e., 
\[ \phi = \sum_{n\geq 0} a_n\, q^n \qquad \text{for some } a_n\in \C.\]
In order to obtain \emph{modular forms} Eisenstein considered the series
\[ \frac{1}{2}\sum_{\substack{\omega \in \Z\tau+\Z\\ \omega\neq 0}} \frac{1}{\omega^k} \qquad (k\geq 3).\]
These \emph{Eisenstein series} are non-zero modular forms of weight~$k$ whenever~$k$ is even---equal to $\mathbb{G}_k$ in~\eqref{eq:eis}. In case $k=2$, we write $\omega=m\tau+n$ and replace the sum over~$\omega$ by $\sum_{m} \sum_{n}\mspace{1mu}$, where $(m,n)\neq (0,0)$. Then, the resulting series equal $\mathbb{G}_2$ and
\[(\mathbb{G}_2|_2\gamma)(\tau) = \mathbb{G}_2 - \pi \mathrm{i} \frac{c}{c\tau+d}\,.\]
Hence, $\mathbb{G}_2$ is not a modular form. 

In fact, the transformation of $\mathbb{G}_2$ leads us to the definition of a \emph{quasi}modular form, which are functions satisfying a more general modular transformation property than the Eisenstein series and of which~$\mathbb{G}_2$ is the first non-modular example. 
\begin{definition} A \textit{quasimodular form} of \textit{weight}~$k$ and \textit{depth at most~$p$} for~$\Gamma$ is a function $\phi:\HH\to \C$ such that 
\begin{enumerate}[{\upshape (i)}]
\item there exist~$\phi_0,\ldots, \phi_p$ so that for all~$\tau\in \HH$ and all~$\gamma = \left(\begin{smallmatrix} a & b \\ c & d\end{smallmatrix}\right) \in \Gamma$ one has
\begin{align}\label{eq:modtrans}(\phi|_k\gamma)(\tau) = \phi_0(\tau) +\phi_1(\tau)\frac{c}{c\tau+d}+\ldots+\phi_p(\tau)\Bigl(\frac{c}{c\tau+d}\Bigr)^{\!p}.\end{align}
\item $\phi$ admits a holomorphic $q$-expansion around infinity $(q=e^{2\pi \mathrm{i} \tau})$, i.e., 
$\phi = \sum_{n\geq 0} a_n q^n$ for some $a_n\in \C.$
Similarly, $\phi_0,\ldots,\phi_p$ are required to admit such an expansion. 
\end{enumerate}
\end{definition}
Note that if~$\phi$ is a quasimodular form, the functions~$\phi_0,\ldots, \phi_p$ are quasimodular forms uniquely determined by~$\phi$ (the function~$\phi_r$ has weight~$k-2r$ and depth~$\leq p-r$)~\cite{Zag08}. Quasimodular forms of depth~$0$ are modular forms. The Eisenstein series~$\mathbb{G}_2$ is an example of a quasimodular form of weight $2$ and depth $1$. 

The differential operator
\[D:=\frac{1}{2\pi\mathrm{i}} \frac{\mathrm{d}}{\mathrm{d}\tau} = q\frac{\mathrm{d}}{\mathrm{d}q}\]
preserves the space of quasimodular forms (cf., \eqref{eq:classicalgkrelations}. Recall the $G_k$ and $\mathbb{G}_k$ differ by a constant, as defined by \eqref{eq:Gk}.).  Besides~$D$, an important differential operator on quasimodular forms is the operator~$\delta$ defined by~$\phi\mapsto 2\pi \mathrm{i}\,\phi_1$ (with~$\phi_1$ defined by the quasimodular transformation property~(\ref{eq:modtrans})). Then, ${\delta G_2=-\frac{1}{2}}$ and, in fact, this property together with the fact that~$\delta$ annihilates modular forms defines~$\delta$ completely. Lastly, one quasimodular forms one defines $W$ as the diagonal operator multiplying a form $\phi$ by its weight. The algebra of quasimodular forms is an $\sltwo$-algebra with respect to these three operators. 

A polynomial in the derivatives of two modular forms may be  modular~\cite{Ran56}. This is the case for the \emph{Rankin--Cohen brackets} of modular forms~$f$ and~$g$ of weight~$k$ and~$\ell$ respectively, defined by~\cite{Coh75}
\begin{align}\label{eq:rcb} [f,g]_n := \sum_{\substack{r,s\geq 0\\r+s=n}}(-1)^r\,\binom{k+n-1}{s}\,\binom{l+n-1}{r}\,D^r\!f\,D^s\!g \quad \quad (n\geq 0).\end{align}
In fact, one can show that for every $\sltwo$-algebra~$A$, the Rankin--Cohen bracket of $f,g\in \ker \delta$ satisfies
\[ \delta [f,g]_n = 0.\]
Here, the Rankin--Cohen bracket $[f,g]_n$ is defined by the same formula~\eqref{eq:rcb}, where, now, $D$ is the derivation of weight~$2$ on $A$.

\begin{example}\label{ex:atildesl2act}
We now give another example of two $\sltwo$-algebras closely related to the $\sltwo$-action defined on $\fmes$ in the next section. For this, we extend our alphabet $\mathcal{A}$ to the set
\begin{align*}
\tA &\df \left\{ \ai{k}{d} \mid k,d \in \Z \right\} \,.
\end{align*}
On $\tA$, we define $\diamond$ as before, i.e., for $k_1,k_2,d_1,d_2\in \Z$ set
\begin{align*}
    \ai{k_1}{d_1} \diamond \ai{k_2}{d_2} = 	\ai{k_1+k_2}{d_1+d_2}\,.
\end{align*}
We obtain an algebra $(\Q \mathcal{\tA},\diamond)$ which is, in contrast to $(\Q \A,\diamond)$, unital with neutral element $\ai{0}{0}$. On $\Q \mathcal{\tA}$ we define for $k,d\geq 0$ the following three linear maps
\begin{align*}
    \tDd:\ai{k}{d} \mapsto k \ai{k+1}{d+1},\qquad   \tWd: \ai{k}{d} \mapsto (k+d) \ai{k}{d},\qquad  \tdeerd: \ai{k}{d} \mapsto d \ai{k-1}{d-1}.  
\end{align*}
By a direct calculation, one can check that these maps are derivations on $(\Q \mathcal{\tA},\diamond)$ and satisfy the commutator relations~\eqref{eq:commutationsrelations}, i.e. $(\Q \mathcal{\tA},\diamond)$ is an $\sltwo$-algebra. Denote by $\ast$ the quasi-shuffle product on~$\Q \mathcal{\tA}$ obtained by~$\diamond$. By Proposition~\ref{prop:derisder}, we obtain derivations $\tD,\tW,\tdeer$ on the quasi-shuffle algebra $(\Q\langle \tA \rangle, \ast)$. These are explicitly given by 
\begin{align*}
    \tD: \ai{k_1,\dots,k_r}{d_1,\dots,d_r} &\mapsto \sum_{j=1}^{r} k_j \ai{k_1,\dots,k_j+1,\dots,k_r}{d_1,\dots,d_j+1,\dots,d_r},\\
    \tdeer: \ai{k_1,\dots,k_r}{d_1,\dots,d_r} &\mapsto \sum_{j=1}^{r} d_j \ai{k_1,\dots,k_j-1,\dots,k_r}{d_1,\dots,d_j-1,\dots,d_r},
\end{align*}
and $\tW$ is the multiplication by $k_1+\dots+k_r+d_1+\dots+d_r\mspace{1mu}$. Again, these derivations satisfy the commutator relations~\eqref{eq:commutationsrelations} and $(\Q\langle \tA \rangle, \ast)$ becomes an $\sltwo$-algebra. In the next section, we will define similar derivations on the subalgebra $(\QA, \ast)$ and show that these are $\sigma$-equivariant. This will give derivations on the quotient $\fmes$. In fact, the derivations $\D$ and $\W$ will just be given by the restrictions of $\tD$ and $\tW$ to the subspace~$\QA$. But notice that the derivation $\tdeer$ is not defined on $(\QA, \ast)$ and we will need to correct it in a certain sense. Since $(\Q\langle \tA \rangle, \qsh)$ seems to be a natural $\sltwo$-algebra, one can ask if there is a natural extension of $\sigma$ to $\Q\langle \tA \rangle$, such that $\tD,\tW$ and $\tdeer$ become $\sigma$-equivariant.
\end{example}

\subsection{\texorpdfstring{$\sltwo$}{sl2}-action on \texorpdfstring{$\fmes$}{FMES}}
In this section, we introduce three derivations $D,W$ and $\deer$ on $\fmes$ and show that these give $\fmes$ the structure of an $\sltwo$-algebra. 
First, we will define the maps $D,W$ and $\delta$ on~$\B$. By Definition~\ref{obs:bimould} they restrict to maps $\QA\to\QA$ and we show they are $\sigma$-equivariant derivations. 

\begin{definition} The $\Q$-linear maps
$D$ and $W$ on $\B$ are defined by $(f_r) \mapsto (Df_r)$ and $(f_r)\mapsto (Wf_r)$ respectively, where
\begin{align}
    \D f_r(\X,\Y)\df \sum_{j=1}^r \frac{\partial}{\partial X_j}\frac{\partial}{\partial Y_j} f_r(\X,\Y)
\end{align}
and 
\begin{align}
    \W f_r(\X,\Y) \df \sum_{j=1}^r\left( \frac{\partial}{\partial X_j}X_j + Y_j \frac{\partial}{\partial Y_j}\right)  f_r(\X,\Y)
\end{align}
The maps $D$ and $W$ restrict to  $(\QA,\ast)$ (see Definition~\ref{obs:bimould}).
\end{definition}
\begin{proposition}
The maps $D$ and $W$ are $\sigma$-equivariant derivations on $(\QA,\ast)$.
\end{proposition}
\begin{proof}
First notice that on the level of coefficients the maps $D,W$ are given by 
\begin{align}\label{eq:Donwords}
    \D 	\ai{k_1,\dots,k_r}{d_1,\dots,d_r}  = \sum_{j=1}^r k_j    \,	\ai{k_1,\dots,k_j+1,\dots,k_r}{d_1,\dots,d_j+1,\dots,d_r}
\end{align}
and 
\begin{align}
    \W	\ai{k_1,\dots,k_r}{d_1,\dots,d_r}  = (k_1+\dots+k_r+d_1+\dots+d_r) \,\ai{k_1,\dots,k_r}{d_1,\dots,d_r}.
\end{align}
That $\D$ and $\W$ are derivations therefore follows from Proposition~\ref{prop:derisder} by noticing that $\ai{k}{d} \mapsto k \ai{k+1}{d+1}$ and $\ai{k}{d} \mapsto (k+d) \ai{k}{d}$, similar as in Example~\ref{ex:atildesl2act}, are both derivations on $(\Q \A, \diamond)$. Now, as 
\[ \sum_{j=1}^r \frac{\partial}{\partial X_j}\frac{\partial}{\partial Y_j} \qquad \text{and} \qquad \sum_{j=1}^r\frac{\partial}{\partial X_j}X_j + Y_j \frac{\partial}{\partial Y_j}\]
are invariant under the change of coordinates in $\sigma$, i.e., under
\[    \bi{X_1,\dots,X_r}{Y_1,\dots,Y_r}   \mapsto \bi{Y_1 + \dots + Y_r,\dots,Y_1+Y_2,Y_1}{X_r,X_{r-1}-X_r,\dots,X_1-X_2}\,,\]
it follows that $D$ and $W$ commute with $\sigma$.
\end{proof}

At first sight, the following definition of $\deer$ might seem unnatural. Note, however, that our numerical computations in Sage suggest that $\deer$ is the unique $\sigma$-equivariant derivation on $\QA$ of weight $-2$ (up to multiplication by a constant). Moreover, as we prove in the next results, this derivation forms an $\sltwo$-triple together with the derivations $D$ and $W$, which (as we explain in the next section) give rise to the usual $\sltwo$-action on the subspace of $\fmes$ of formal quasimodular forms. 
\begin{definition} Define the linear map $\deer: \B\rightarrow \B$ by $\deer(f_r)=(\deer f_r)$ with $ \deer f_r(\X,\Y)$ given by
\begin{align*}
 & (X_1Y_1+\ldots+X_rY_r)f_r(\X,\Y)  \\&\qquad + \sum_{j\geq 1} \biggl(-(X_j - X_{j+1} + Y_{j}) \frac{\varphi_{j}^{+}}{2}-(X_{j-1} - X_{j} + Y_{j}) \frac{\varphi_{j}^{-}}{2} 
 +\Bigl(\frac{\varphi_{j}^+}{2}\Bigr)^2- \Bigl(\frac{\varphi_{j}^{-}}{2}\Bigr)^2\biggr)f_r(\X,\Y).
\end{align*}
By Definition~\ref{obs:bimould} we obtain a $\Q$-linear map $\delta: \QA \rightarrow \QA$.
\end{definition} 

\begin{proposition}\label{prop:deer}
 The map~$\deer$ is a $\sigma$-equivariant derivation on $(\QA,\ast)$.
\end{proposition}
\begin{proof}
The fact that $\deer$ commutes with $\sigma$ follows directly by observing $X_1Y_1+\ldots +X_rY_r$ is invariant under the change of coordinates dictated by $\sigma$ and by making use of the relations~\eqref{eq:phiandsigma}. 

We now show that $\delta$ is a derivation. For this notice that on the level of coefficients, $\delta$ is explicitly given by
\begin{align*}
 \deer \df \deer^1 -\frac{1}{2}\bigl( \deer^2 + \deer^3 + \deer^4 + \deer^5\bigr)\mspace{1mu},
\end{align*}
where $\deer^j: \QA \rightarrow \QA$ ($1\leq j \leq 5$) are the linear maps defined on the generators by
\begin{align*}
\deer^1 \ai{k_1,\dots,k_r}{d_1,\dots,d_r} &\df \sum_{j=1}^r \ind_{k_j>1}\,  d_j \, \ai{k_1,\dots,k_j-1,\dots,k_r}{d_1,\dots,d_j-1,\dots,d_r}\\
&\qquad-\frac{1}{2}\sum_{j=1}^{r-1} \ind_{\substack{k_{j+1}=1 }} \,d_{j+1}\,\ai{k_1,\dots,k_j,k_{j+2},\dots,k_r}{d_1,\dots,d_j+d_{j+1}-1,d_{j+2},\dots,d_r} \\
&\qquad -\frac{1}{2} \sum_{j=1}^{r-1} \ind_{k_j=1} \,d_j\, \ai{k_1,\dots,k_{j-1},k_{j+1},\dots,k_r}{d_1,\dots,d_{j-1},d_j+d_{j+1}-1,\dots,d_r},\\[5pt]
\deer^2 \ai{k_1,\dots,k_r}{d_1,\dots,d_r} &\df \ind_{\substack{k_r=2\\d_r=0}} \,\ai{k_1,\dots,k_{r-1}}{d_1,\dots,d_{r-1}}\,-\,\frac{1}{2}\, \ind_{\substack{k_{r-1}=k_{r}=1\\d_{r-1}=d_r=0}}\, \ai{k_1,\dots,k_{r-2}}{d_1,\dots,d_{r-2}},\\[5pt]
\deer^3 \ai{k_1,\dots,k_r}{d_1,\dots,d_r} &\df  \ind_{\substack{k_r=1\\d_r=1}}\, \ai{k_1,\dots,k_{r-1}}{d_1,\dots,d_{r-1}},\\[5pt]
\deer^4 \ai{k_1,\dots,k_r}{d_1,\dots,d_r} &\df  \sum_{j=2}^{r} \kd{\!\!\!k_{j}=1}{\!\!\!k_{j-1}>1}\, \ai{k_1,\dots,k_{j-1}-1,k_{j+1},\dots,k_r}{d_1,\dots,d_{j-1}+d_{j},d_{j+1},\dots,d_r}\,
\\&\qquad-\sum_{j=1}^{r-1} \kd{\!\!\!k_j=1}{\!\!\!k_{j+1}>1}\,  \ai{k_1,\dots,k_{j-1},k_{j+1}-1,\dots,k_r}{d_1,\dots,d_{j-1},d_j+d_{j+1},\dots,d_r},%\\
 \end{align*}
 \begin{align*}
\deer^5 \ai{k_1,\dots,k_r}{d_1,\dots,d_r} &\df\sum_{j=1}^{r-1} \ind_{k_j=2}\, \ai{k_1,\dots,k_{j-1},k_{j+1},\dots,k_r}{d_1,\dots,d_{j-1},d_j+d_{j+1},\dots,d_r} \,\\&\qquad-\sum_{j=1}^{r-1} \ind_{k_{j+1}=2}\, \ai{k_1,\dots,k_j,k_{j+2},\dots,k_r}{d_1,\dots,d_j+d_{j+1},d_{j+2},\dots,d_r}\\
&\qquad+\frac{1}{2} \sum_{j=1}^{r-2} \ind_{k_{j+1}=k_{j+2}=1} \,\ai{k_1,\dots,k_j,k_{j+3},\dots,k_{r}}{d_1,\dots,d_j+d_{j+1}+d_{j+2},d_{j+3},\dots,d_r}\\
&\qquad-\frac{1}{2} \sum_{j=1}^{r-2} \ind_{k_j=k_{j+1}=1}\,\ai{k_1,\dots,k_{j-1},k_{j+2},\dots,k_r}{d_1,\dots,d_{j-1},d_j+d_{j+1}+d_{j+2},\dots,d_r}.
\end{align*} 
Now, we show each $\deer^i$ is a derivation:
\begin{enumerate}[(i)]
\item For $\deer^1$ use the derivation $\dword$ of Lemma~\ref{lem:thetaphideriv} with $\dletter: \ai{k}{d} \mapsto \ind_{k>1} d \ai{k-1}{d-1}$. Note 
\[
\gamma\biggl(\ai{k_1}{d_1},\ai{k_2}{d_2}\biggr) = \bigl(\ind_{k_1=1} d_1+\ind_{k_2=1}d_2\bigr)\ai{k_1+k_2-1}{d_1+d_2-1}.
\]
\item For $\deer^2$ use the derivation $\drighta$ of Corollary~\ref{cor:removeletterderqsh} with $a=\ai{2}{0}$.
\item For  $\deer^3$ also use the derivation $\drighta$ of Corollary~\ref{cor:removeletterderqsh}, now with $a=\ai{1}{1}$.

\item For $\deer^4$ consider the derivation $\Theta_S$ of Lemma~\ref{lem:dphiaonqsh}, where $S=\{ \ai{1}{d'} \mid d'\geq 0\}$ and for $a=\ai{1}{d'}$, we let $\varphi_{a}: \ai{k}{d} \mapsto \ind_{k>1} \ai{k-1}{d+d'}$. 
\item Using Proposition~\ref{prop:phideriv}, we find
\[\delta^5 = \sum_{\substack{d\geq 0 \\a=\ai{2}{d}}} \exp_\diamond \circ\, \Theta^{\varphi_a,a} \circ \, \log_\diamond,\]
where $\varphi_a: \ai{k'}{d'} \mapsto \ai{k'}{d+d'}$ for $a=\ai{2}{d}$.
Namely, the letter $\ai{2}{d}$ appears as a $\diamond$-product of two letters: $\ai{2}{d} = \ai{1}{d_1} \diamond \ai{1}{d_2}$ for any $d_1+d_2=d$. Therefore,  $\Theta^{\varphi_a,a}$ does not commute with $\log_{\diamond}\mspace{1mu}$, since in the definition of $\log_{\diamond}$ in~\eqref{eq:deflog} we could have for some $1\leq j \leq l$ that $i_j = 2$ with $a=a_{i_1+\dots+i_j+1} \diamond  a_{i_1+\dots+i_j+2}\mspace{1mu}$. This gives rise to the additional terms with factor $\frac{1}{2}$ in~$\delta^5$. \qedhere
\end{enumerate}
\end{proof}

\begin{theorem}\label{thm:fmesissl2algebra}
The space~$\fmes$ is an $\sltwo$-algebra.
\end{theorem}
\begin{proof}
By the previous propositions, it remains to show that the derivations $D,W,\deer \in \operatorname{Der}(\fmes)$  form an \emph{$\sltwo$-triple}. Using the explicit formulas for $D, W$ in~\eqref{eq:Donwords} and for $\delta^i$ in the proof of Proposition~\ref{prop:deer}, one directly find $[D,W]=-2D$. Also,  it is not hard to verify that 
\[ [\deer^i,W] = 2\deer^i \qquad (i=1,2,3,4,5)\]
and
\[ [\deer^1,D] = W \qquad \text{and} \qquad [\deer^i,D]=0 \qquad (i=2,3,4,5). \qedhere\]
\end{proof}

\begin{remark}\label{rem:uniquedwdelta}
The derivations $D,W,\omega$ and $\deer$ are of weight $2,0,-1,-2$ respectively. 
Numerical experiments in Sage suggest that these are the only $\sigma$-equivariant derivations on $(\QA, \ast)$ of weight between $-2$ and $2$ (up to multiplication by a constant). In particular, there seems not to exist a non-trivial $\sigma$-equivariant derivation of weight~$1$.
\end{remark}

\begin{remark}
The map $[\omega,\deer]=\omega \deer -\deer \omega$ is a non-trivial $\sigma$-equivariant derivation of weight $-3$. Moreover, the map $\ttt:\B\to \B$ defined by $\ttt(f_r)=(\ttt f_r)$ conjecturally yields a $\sigma$-equivariant derivation of weight $-3$, where $\ttt$ acts on $f_r$ by
\begin{align*}
\ttt \df \, 
&\sum_{j=1}^r (X_j-X_{j+1}+ Y_{j} )^2 \,\varphi_{j}^{+} - \sum_{j=2}^r (X_{j-1} - X_{j}  - Y_j)^2 \,\varphi_{j}^{-} \\
&+\sum_{j=1}^{r-1} (-X_j+X_{j+2} - Y_{j}- Y_{j+1} )\, (\varphi_{j}^{+})^2 \\
&-\sum_{j=2}^{r-1} (X_{j-1}-4X_j  + 3X_{j+1} - 3Y_j + Y_{j+1} ) (\varphi_{j}^{-})^2 \\
&+\frac{1}{3}\sum_{j=1}^{r-2} (\varphi_{j}^{+})^3-\frac{1}{3}\sum_{j=2}^{r-2} (\varphi_{j}^{-})^3.
\end{align*}
In particular, $[\omega,\deer]$ and $\ttt$ seem to form a basis for the space of derivations of weight $-3$. Explicitly, on elements in $\fmesz$ the derivation $\ttt$ is given by
\begin{align*}
\ttt\, & \gf(k_1,\ldots,k_r) =\\
& \sum_{j=1}^r 
(-1)^{k_j+1}\,\binom{2}{k_j-1} \,\gf(k_1,\ldots,k_{j-1},k_j+k_{j+1}-3,\ldots,k_r) \\ 
& -\sum_{j=2}^r 
(-1)^{k_j+1}\,\binom{2}{k_j-1} \,\gf(k_1,\ldots,k_{j-1}+k_j-3,k_{j+1},\ldots,k_r) \\
&+\sum_{j=1}^{r-1} \ind_{(k_j,k_{j+1})\in \{(1,1),(2,1)\}}\,(-1)^{k_j+1} \,\gf(k_1,\ldots,k_{j-1},k_j+k_{j+1}+k_{j+2}-3,\ldots,k_r) \\
&-\sum_{j=2}^{r-1} (\ind_{(k_j,k_{j+1})=(1,1)}-4\cdot\ind_{(k_j,k_{j+1})=(2,1)}+3\cdot\ind_{(k_j,k_{j+1})=(1,2)})\\[-5pt]&\hspace{180pt}\gf(k_1,\ldots,k_{j-1}+k_j+k_{j+1}-3,k_{j+2},\ldots,k_r)\\
&+\frac{1}{3}\ind_{(k_1,k_{2},k_{3})=(1,1,1)}\,
\gf(k_{4},\ldots,k_r),
\end{align*}
where $\gf$ of a sequence containing non-positive entries is considered to be $0$. It remains an interesting open question to find all $\sigma$-equivariant derivations on $\QA$. 
\end{remark}

\subsection{Restriction to \texorpdfstring{$\fmesz$}{FMES} and the space~\texorpdfstring{$\fames$}{E}}
We provide explicit formulas for the derivations~$D$ and~$\deer$ on~$\fmesz$ (the subspace which conjecturally equals $\fmes$). In particular, these formulas imply the following.

\begin{proposition}\label{prop:sltwosubalgebra}
The subspace~$\fmesz$ is an $\sltwo$-subalgebra of $\fmes$. 
\end{proposition}
\begin{proof}
From the explicit formulas in the proof of Proposition~\ref{prop:deer}, we directly see that $\deer$ keeps the space~$\fmesz$ invariant:
\begin{align}\begin{split}\label{eq:deltainfil0}
    \deer\, \gf(k_1,\dots,k_r) & = -\frac{1}{2} \ind_{k_1=2} \,\gf(k_2,\dots,k_r) + \frac{1}{4} \ind_{k_1=k_2=1} \gf(k_3,\dots,k_r)\\
    &\qquad +\frac{1}{2}\sum_{j=1}^{r-1}\kd{\!\!\!k_j=1}{\!\!\!k_{j+1}>1}\, \gf(k_1,\dots,k_{j-1},k_{j+1}-1,\dots,k_r) \\
    &\qquad - \frac{1}{2}  \sum_{j=2}^{r}\kd{\!\!\! k_j=1}{\!\!\! k_{j-1}>1}\, \gf(k_1,\dots,k_{j-1}-1,k_{j+1},\dots,k_r).
    \end{split}
\end{align}
Moreover, following the same proof as in~\cite[Proposition~6.30 and Corollary~6.31]{BB} we have
\begin{align}\label{eq:Donfil0}
    \D  \gf(k_1,\dots,k_r) = \gf\bigl( z_2 \ast z_{k_1} \cdots z_{k_r} - z_2 \shuffle z_{k_1} \cdots z_{k_r} \bigr),
\end{align}
where the right-hand side is given by the map $\eqref{eq:Gffromh1}$ and $\shuffle$ denotes the shuffle product on $\h^1$ defined in the introduction. 
Hence, the action of the $\sltwo$-triple restricts to $\fmesz$. 
\end{proof}

\begin{remark} 
\emph{Define}  $\delta$ and $\D$ on $\h^1$ by the formulas~\eqref{eq:deltainfil0} and~\eqref{eq:Donfil0}.
It might be interesting to consider the algebra $(\h^1,\ast)$ modulo the relations which ensure that it becomes an $\sltwo$-algebra. For example, since $D: w \mapsto z_2 \ast w - z_2 \shuffle w$ is not a derivation on $(\h^1,\ast)$, we would impose the relation $D(w \ast v) - D(w) \ast v - w \ast D(v) =0$ for all $w,v \in \h^1$. In the special case $v=w=z_1\mspace{1mu}$, one finds $z_4 - 2 z_2 z_2 + 2 z_3 z_1=0$. This corresponds to the first relation among the elements in~$\fmesz$, i.e., $\gf(4) = 2 \gf(2,2) - 2 \gf(3,1)$. The relations obtained this way might be a first step in answering the question raised in Remark~\ref{rem:anotherdefoffmes}, namely, to find an ideal $\mathfrak{R} \subset (\h^1,\ast)$, such that $\fmes \cong \faktor{(\h^1,\ast)}{\mathfrak{R} }$.
\end{remark}

As a formal analogue of the space~$\ames$ defined in~\eqref{eq:defames}, we define the following subspace of~$\fmesz$
\begin{align}\label{eq:deffames}
    \fames \df \langle \gf(k_1,\dots,k_r) \mid  r\geq 0, k_1,\dots,k_r \geq 2\rangle_\Q\,.
\end{align}
By the definition of the stuffle product, we see that $\fames$ is a subalgebra. Clearly, the algebra~$\fames$ is closed under $W$ and by~\eqref{eq:deltainfil0} it is also closed under $\deer$. But from~\eqref{eq:Donfil0} we do not get immediately that $\fames$ is closed under $D$, since, for example, we have $$z_2 \ast z_3 - z_2 \shuffle z_3 = (z_2 z_3 + z_3 z_2 + z_5) - (z_{2}z_{3}+ 3 z_3 z_2 + 6 z_4 z_1)= z_5 - 2 z_3 z_2 - 6 z_4 z_1$$ and therefore
\begin{align}
    D \gf(3) = \gf(5) - 2 \gf(3,2) - 6\gf(4,1).
\end{align}
Here $\gf(4,1)$ is, a priori, not an element of $\fames$, but we can give the following alternative expression for $D \gf(k)$, which shows that $D \gf(k) \in \fames$ for all $k\geq 2$.

\begin{proposition}
For $k\geq 2$ we have 
\begin{align*}
    D \gf(k) = (2k-1) \gf(k+2) - \gf(2,k) - \sum_{j=2}^k (k+j-1)\, \gf(k+2-j,j).
\end{align*}
\end{proposition}
\begin{proof}
Using Proposition~\ref{prop:dshindep1} for $d_1=d_2=0$ gives 
for $k_1,k_2\geq 1$ with $k_1+k_2\geq 3$
\begin{align}
\frac{(k_1+k_2-3)!}{(k_1-1)!(k_2-1)!}D \gf(k_1+k_2-2)&=\gf(k_1,k_2)+\gf(k_2,k_1)+\gf(k_1+k_2)\\&\qquad -\sum_{j=1}^{k_1+k_2-1}\!\left(\!\binom{j-1}{k_1-1}+\binom{j-1}{k_2-1}\!\right)\gf(j,k_1+k_2-j).
\end{align}
The formula in the proposition follows by taking $2k$-times the case $k_1=1, k_2=k+1$ and subtracting the case $k_1=2, k_2=k$.
\end{proof}

\begin{conjecture}\label{conj:espacesl2}
\begin{enumerate}[{\upshape(i)}]
\item The algebra~$\fames$ is an $\sltwo$-subalgebra of $\fmes$.
\item The $\Q$-linear map defined by 
    \begin{align*}
    \fames &\longrightarrow \ames\\
    \gf(k_1,\dots,k_r) &\longmapsto \mathbb{G}_{k_1,\dots,k_r}
    \end{align*}
    is an algebra isomorphism. In particular, $\ames$ is also an $\sltwo$-algebra.
\end{enumerate}   
\end{conjecture}
Note that by~\eqref{eq:deltainfil0}, for (i) it suffices to show that $\fames$ is closed under the derivation~$D$.

\begin{remark} Notice that the $\sltwo$-structure on $\fmes$ does not restrict to the quotient, i.e., the space of formal multiple zeta values $ \fmz = \faktor{\fmes\,}{\mathfrak{N}}$ is not an $\sltwo$-algebra for the restricted operators. The derivation~$D$ is well-defined on the quotient and becomes the zero map on this quotient (c.f., Proposition~\ref{prop:Dinkerphi}), but the derivation~$\delta$ is not well-defined. For example, $\delta \gf\bi{2}{1} = \gf \bi{1}{0}$ and $\pi \gf\bi{2}{1}=0$, but $\fzeta(1) = \pi \gf \bi{1}{0}  \neq 1$.
\end{remark}

\section{Formal quasimodular forms}\label{sec:formalquasimodularforms}
In this section, we will be interested in a subalgebra of $\fmes$ isomorphic to the $\sltwo$-algebra of quasimodular forms:
\begin{definition}	We define the algebra of \emph{formal quasimodular forms $\fqmf$} as the smallest $\sltwo$-subalgebra of $\fmes$ which contains $\gf(2)$.
\end{definition}
 \begin{remark}
By Corollary~\ref{cor:formalgkrelations}, $\fqmf$ is a \emph{non-trivial} proper finitely generated $\sltwo$-subalgebra of $\fmes$. Are there more non-trivial subalgebras?  The first subalgebra one could consider is $\Q[ \gf(k) \mid k\geq 1]$.
But one can check that this algebra is not closed under the operator $\D$; for example, it does not contain $\D \gf(1)$. Secondly, note that by Proposition~\ref{prop:sltwosubalgebra} the algebra~$\fmesz$ is also an $\sltwo$-subalgebra. Conjecturally, this subalgebra is not a proper subalgebra (see Conjecture~\ref{conj:BBvI}).
In addition to the algebra~$\fames$ defined in~\eqref{eq:deffames}, the authors were unable to find other possible candidates for the $\sltwo$-subalgebras. In particular, it would be interesting to know if $\fqmf$ might be the only non-trivial finitely generated $\sltwo$-subalgebra of $\fmes$.
\end{remark}
As a consequence of the relations we proved in Section~\ref{sec:rel}, we obtain the following.

\begin{corollary}
The Ramanujan differential equations are satisfied in $\fmes$, i.e.,
\begin{align*}
    D\gf(2) &= 5 \gf(4) - 2 \gf(2)^{ 2}\,,\\
    D\gf(4) &= 14 \gf(6) - 8 \gf(2) \, \gf(4)\,,\\
    D\gf(6) &= \frac{120}{7} \gf(4)^2 - 12 \gf(2)\, \gf(6)\,. 
\end{align*}
\end{corollary}
\begin{proof}
Recall Corollary~\ref{cor:mfprod}(i), which is equivalent to 	
\begin{align*}
    D\gf(k) = \frac{k(k+3)}{2}\gf(k+2)- k \sum_{\substack{k_1+k_2=k+2 \\ k_1, k_2 \text{ even} }} \gf(k_1) \gf(k_2)\,.
\end{align*}
The formulas for $D\gf(2)$ and $D\gf(4)$ follow directly from this and the formula for $	D\gf(6)$ follows after using $\gf(8) = \frac{6}{7} \gf(4)^2$, which is a consequence of Corollary~\ref{cor:mfprod}(ii).
\end{proof}

\begin{theorem}\label{thm:fqmfcongqmf} We have $\fqmf \cong \qmf$ as $\sltwo$-algebras.
\end{theorem}
\begin{proof}
By the previous corollary, clearly $\gf(2),\gf(4), \gf(6)$ are elements of $\fqmf$. Moreover, $\gf(2), \gf(4)$ and $\gf(6)$ are algebraically independent over $\Q$ as follows from Theorem~\ref{thm:cmes} and the well-known fact that this is true for $G(2),G(4),G(6) \in \qmf \subset \Q\llbracket q \rrbracket$. Hence, the algebra $\Q[\gf(2),\gf(4), \gf(6)]$ is free and contained in $\fqmf$. 
Since this is also an $\sltwo$-subalgebra of $\fmes$, on which the action by the $\sltwo$-triple $D,W,\deer$ is the same, the result follows. 
\end{proof}

The following results on formal (quasi)modular forms follow directly from the previous result, as these statements are known to be true for (quasi)modular forms. Nevertheless, we want to formulate these results in the language of formal objects and give a self-contained proof. For example, the following Corollary follows from Corollary~\ref{cor:mfprod}. 
\begin{corollary}\label{cor:formalgkrelations} 
The algebra of formal quasimodular forms $\fqmf$ satisfies:
\begin{enumerate}[{\upshape(i)}] 
\item $\gf(k)\in \fqmf$ for all even $k\geq 2$. 
\item 	$\fqmf=\Q[\gf(2),\gf(4),\gf(6)] = \Q[\gf(2),\D \gf(2),\D^2 \gf(2)]$.
\item The Chazy equation is satisfied, i.e.,
    \begin{align}\label{eq:chazy}
        D^3\gf(2) + 24\, \gf(2)\,  D^2\gf(2) - 36 (D\gf(2))^{2} = 0\,.
    \end{align}
\end{enumerate}
\end{corollary}

\begin{definition} We define the algebra of \emph{formal modular forms} by $\fmf = \ker \deer_{\mid \fqmf}$.  Write $\fmf_k$ to denote the space of all formal modular forms of weight $k\geq 0$.
\end{definition}

\begin{proposition}\label{prop:mfisom}
    We have $\fmf=\Q[\gf(4),\gf(6)] \cong \mf$.
\end{proposition}
\begin{proof}
The first isomorphism follows directly from Corollary~\ref{cor:formalgkrelations} after computing $\deer(\gf(4))=0$ and $\deer \gf(6) =0$. The second isomorphism then follows from Theorem~\ref{thm:fqmfcongqmf}.
\end{proof}

\begin{remark} The definition of $\fmf$ raises the question what the kernel of $\delta$ is when considered on the whole space~$\fmes$. By the direct formula in the proof of Proposition~\ref{prop:deer}, we see that it contains all $\gf(k_1,\dots,k_r)$, with $k_1\geq 3, k_2,\dots,k_r\geq 2$; these are the indices for which the classical multiple Eisenstein series~\eqref{eq:defames} converge absolutely. 
In fact, restricted to $\fames$, the elements $\gf(k_1,\dots,k_r)$ with $k_1\geq 3, k_2,\dots,k_r\geq 2$ capture the entire kernel of~$\delta$, and every formal multiple Eisenstein series in $\fames$ 
can be written as a polynomial in~$\gf(2)$ with coefficients given by these elements. This is similar to a property of quasimodular forms, which can be written uniquely as a polynomial in~$\mathbb{G}_2$ with modular coefficients. 
However, note that the kernel of~$\delta$ on $\fmes$ also contains elements with non-admissible indices, for example, $\gf(1,2,1) \in \ker(\delta)$. 
\end{remark}

The following corollary follows from Proposition~\ref{prop:mfisom} and the fact that the corresponding statement is true for modular forms. For modular forms, this result follows from~\cite{KZ} and heavily depends on the analytic theory of modular forms, in particular on Rankin's method. It is an interesting open question to obtain a \emph{combinatorial} proof of the following result, that is, by only using the stuffle product and the swap.  
\begin{corollary} For $k\geq 4$ we have
\begin{align}
        \fmf_k = \Q \gf(k) + \langle \gf(a)
        \,\gf(b) \mid a+b =k, a,b\geq 4 \text{ even} \rangle_\Q \,.
\end{align}
In particular, $\fmf \subset \fqmf \subset \fil{\dep}{2}\fmes$.
\end{corollary}

Recall the projection $\proj: \fmes \rightarrow \fmz$ given by~\eqref{eq:pi}. 
\begin{definition} We define the algebra of $\emph{formal cusp forms}$ by $\fs = \ker \proj_{|\fmf}$. Write $\fs_k$ to denote the space of all formal cusp forms of weight $k\geq 0$.
\end{definition}

The first example of a non-zero formal cusp form appears in weight $12 = \operatorname{lcm}(4,6)$.
\begin{align*}
    \fDelta  &=  \frac{\eis(4)^{3}- \eis(6)^{2}}{1728}= 2400\cdot 6! \cdot \gf(4)^{3} - 420 \cdot 7! \cdot \gf(6)^{2}\,,
\end{align*}
where we write $\eis(k) = - \frac{2 k!}{B_k} \gf(k)$ for even $k\geq 2$. These elements correspond to the Eisenstein series $E_k\mspace{1mu}$, which all have constant term $1$.
\begin{proposition} We have $\fDelta \in \fs_{12}$ and $\D \fDelta = \eis(2) \fDelta$. 
\end{proposition}
\begin{proof}
    Using that
\[      \gf(4) =  \frac{1}{5} D\gf(2)+\frac{2}{5} \gf(2)^{ 2},  \qquad
     \gf(6)=    \frac{1}{70} D^2\gf(2)+ \frac{6}{35} \gf(2) D\gf(2)+\frac{8}{35} \gf(2)^{3},
     \]
we can write $\fDelta$ explicitly as a polynomial in $\gf(2), D\gf(2), D^2\gf(2)$ as follows		
\begin{multline}
        \frac{1}{432} \fDelta = 
        48 \gf(2)^2 (D\gf(2))^2
        +32(D\gf(2))^3
        -32 \gf(2)^3 D^2\gf(2)\\
        -24 \gf(2) (D\gf(2)) (D^2\gf(2)) 
        -(D^2\gf(2))^2 \,.
    \end{multline}
In particular we see that no term with $\gf(2)^6$ appears, i.e., all remaining terms are elements in the ideal $\mathfrak{N}$ (see~\eqref{eq:N}) and therefore $\fDelta \in \ker \proj$. The second statement follows by taking the derivative of this equation and then using the Chazy equation~\eqref{eq:chazy}.
\end{proof}

\begin{proposition}
    We have $\fmf_k= \Q \gf(k) \oplus  \fs_k\mspace{1mu}$.
\end{proposition}
\begin{proof}
    Every element $f \in \fmf_k$ can be written as a polynomial in $\gf(2), D\gf(2), D^2\gf(2)$. Suppose that $f = a \gf(2)^{\frac{k}{2}} + f_1$ for some $a\in \Q$, where $f_1$ contains monomials which contain $D\gf(2)$ or $D^2 \gf(2)$, i.e., $f_1 \in \ker \pi$. By Corollary~\ref{cor:eulerrelation} we have $\gf(k) = - \frac{B_{k}}{2 k!} (-24 \gf(2))^{\frac{k}{2}} + g_1$ for some $g_1\in \ker \pi$, i.e., $\gf(k) \notin \fs_k\mspace{1mu}$.
    We conclude $f + a\frac{2 k!}{B_k } (-24)^{-\frac{k}{2}}\gf(k)\in \fs_k$ from which the statement follows. 
\end{proof}

Analogously, the proof of many other statements known for (quasi)modular forms can be translated to the formal setup. We invite the interested reader or motivated students to do so.

\section{Balanced setup}\label{sec:balanced}
In this section, we will introduce another setup that can be used to describe the algebra~$\fmes$. This setup is motivated by the idea of making the involution $\sigma$ simpler on the level of coefficients as it was done in~\cite{Zu} for the $q$-series $g$ given in~\eqref{def:big}. It was introduced by Burmester in her thesis~\cite{Bu1} and then studied in~\cite{Bu2} and~\cite{Bu3}.
Let $\QB$ be the non-commutative polynomial ring in the alphabet $\mathcal{B}=\{b_0,b_1,\ldots \}$, and let $\mathcal{B}^\ast$ be the free monoid generated by the elements of $\mathcal{B}$, to which we refer to as \emph{words}. On~$\QB$ we recursively define the \emph{quasi-shuffle product}~$\ast_b$ as the $\Q$-bilinear product satisfying
\begin{align*}
    b_i u \ast_b b_j v \=  b_i ( u \ast b_j v)+ b_j (b_i u \ast v) + (b_i \diamond b_j)(u \ast v) \qquad(i,j \geq 0 \text{ and } u,v \in \QB)
\end{align*}
and $1 \ast w = w \ast 1=w$ for any $w \in \QB$. Here $b_i \diamond b_j = \delta_{ij >0}\, b_{i+j}$ is an associative and commutative product on $\Q \mathcal{B}$. 
As before, the space~$(\QB,\ast)$ is a commutative $\Q$-algebra  (see~\cite{H} and also~\cite{HI}).
The subspace~$\QB^0$ of words not starting in $b_0$ is closed under $\ast$ which gives rise to a commutative $\Q$-algebra $(\QB^0,\ast)$. 

 Define  the $\Q$-linear involution $\tau: \QB^0 \rightarrow \QB^0 $ by
\[\tau(b_{k_1} b_0^{m_1} \cdots b_{k_s} b_0^{m_s}) \df b_{m_s+1} b_0^{k_s-1} \cdots b_{m_1+1} b_0^{k_1-1} \qquad (k_1,\ldots,k_s \geq 1 \text{ and } m_1,\dots,m_s\geq 0).\]

We now want to compare the algebra $(\QA,\qsh)$ with $(\Q \langle \mathcal{B}\rangle^0,\ast)$. For this define the $\Q$-linear map $\varphi: \Q \langle \mathcal{B}\rangle^0 \rightarrow \QA$ for all $r\geq 1$ on the generators by 
\begin{align*}
   \sum_{\substack{k_1,\dots,k_r\geq 1\\m_1,\dots,m_r \geq 0}} 	\varphi(b_{k_1} b_0^{m_1}\cdots b_{k_r} b_0^{m_r})\,X_1^{k_1-1} \cdots X_r^{k_r-1} Y_1^{m_1} \cdots Y_r^{m_r} \df     \mathfrak{A}\bi{X_1,\dots,X_r}{Y_1,Y_2-Y_1,\dots,Y_r-Y_{r-1}}\,.
\end{align*}
\begin{proposition}[{\cite[Theorem~6.4]{Bu1}}]
The map $\varphi: (\Q \langle \mathcal{B}\rangle^0,\ast_b) \rightarrow (\QA,\ast)$ is an isomorphism of graded $\Q$-algebras satisfying additionally $\sigma \circ \varphi = \varphi \circ \tau$.
\end{proposition}
As a consequence of this proposition, we get 
\begin{align*}
    \fmes \cong \faktor{ (\QB^0 ,\ast)}{T}\,,
\end{align*}
where $T$ is the ideal in  $(\QB^0 ,\ast)$ generated by $\tau(w)-w$ for all $w\in \QB^0$. In particular, the algebra~$\fmes$ in~\cite{Bu3} agrees with the original definition of the algebra of formal multiple Eisenstein series used here. The operators $D,W,\delta$ can be easily translated into this setup. For example, for a word $w=b_{k_1} b_0^{m_1} \cdots b_{k_r} b_0^{m_r}$ with $k_1,\dots,k_r\geq 1, m_1,\dots,m_r\geq 0$ the operator $D$ is given by
\begin{align*}
    D(w) = \sum_{1 \leq i \leq j \leq r} k_i (m_j + 1)\, b_{k_1} b_0^{m_1} \cdots b_{k_i + 1} \cdots b_{k_j} b_0^{m_j+1} \cdots b_{k_r} b_0^{m_r}.
\end{align*}
\newpage
\appendix
\section{Classical double shuffle and swap \& stuffle}\vspace{-5pt}
{\begin{center}(by Nils Matthes) \end{center}}\vspace{5pt}
To make this appendix self-contained, we will recall some facts on quasi-shuffle algebras which were already mentioned in the Introduction or in Section~\ref{sec:derivationsforqsh}. All algebras are assumed to be commutative, associative, and unital unless stated otherwise. 

\subsection{The extended double shuffle relations}

\newcommand{\ishuffle}{\overline{\shuffle}}

\subsubsection{Product structures on words}
Let $\LL_z=\{z_k \mid k\geq 1\}$ be the alphabet of letters indexed by the positive integers and let $\bQ\langle \LL_z\rangle$ be the free $\bQ$-vector space on the free monoid~$\LL_z^\ast$ of words on~$\LL_z\mspace{1mu}$, including the empty word~$1$. It can be equipped with two product structures: the first is the index shuffle product\footnote{The index shuffle product $\ishuffle$ is the shuffle product on the alphabet $\LL_z\mspace{1mu}$. One should not confuse it with the shuffle product $\shuffle$ on the alphabet $\LL_{xy}\mspace{1mu}$.} which is defined for words $w,w_1,w_2 \in \LL_z^\ast$ and letters $\sy_{k_1},\sy_{k_2} \in \LL_z$ by the recursive formula
\[
\begin{aligned}
\sy_{k_1}w_1\ishuffle \sy_{k_2}w_2&=\sy_{k_1}(w_1\ishuffle \sy_{k_2}w_2)+\sy_{k_2}(\sy_{k_1}w_1\ishuffle w_2)\\
1\ishuffle w&=w\ishuffle 1=w.
\end{aligned}
\]
The second is the stuffle product, which is defined similarly to before, but with one extra term:
\[
\begin{aligned}
\sy_{k_1}w_1\ast \sy_{k_2}w_2&=\sy_{k_1}(w_1\ast \sy_{k_2}w_2)+\sy_{k_2}(\sy_{k_1}w_1\ast w_2)+\sy_{k_1+k_2}(w_1\ast w_2)\\
1\ast w&=w\ast 1=w.
\end{aligned}
\]
Each of these products turns $\bQ\langle \LL_z\rangle$ into a $\bQ$-algebra whose unit element is the empty word in both cases. Together with the deconcatenation coproduct
\[
\begin{aligned}
\Delta: \bQ\langle \LL_z\rangle&\rightarrow \bQ\langle \LL_z\rangle \otimes \bQ\langle \LL_z\rangle\\
\sy_{k_1}\ldots\sy_{k_n} &\mapsto
\sum_{i=0}^{n}\sy_{k_1}\ldots\sy_{k_i}\otimes \sy_{k_{i+1}}\ldots\sy_{k_n},
\end{aligned}
\]
both $(\bQ\langle \LL_z\rangle,\ishuffle,\Delta)$ as well as $(\bQ\langle \LL_z\rangle,\ast,\Delta)$ are commutative Hopf algebras.

\subsubsection{Commutative power series}

Let $R$ be a commutative unital $\bQ$-algebra. It is well known that $\bQ$-linear morphisms $\varphi \in \Hom(\bQ\langle \LL_z\rangle,R)$ can be identified with tuples of power series. More precisely, there is an $R$-linear isomorphism
\begin{equation} \label{eqn:trafo}
\begin{aligned}
\Hom(\bQ\langle \LL_z\rangle,R)&\rightarrow \prod_{n\geq 0}R\llbracket X_1,\ldots,X_n\rrbracket\\
\varphi&\mapsto (\Phi_n)_{n\geq 0},
\end{aligned}
\end{equation}
where $\Phi_0=\varphi(1)$ and $\Phi_n(X_1,\ldots,X_n)=\sum_{k_1,\ldots,k_n\geq 1}\varphi(\sy_{k_1}\ldots\sy_{k_n})\,X_1^{k_1-1}\ldots X_n^{k_n-1}$, for $n\geq 1$. Under this correspondence, the convolution product $\varphi \mspace{1mu}\star\mspace{1mu} \psi:=m_R\circ(\varphi\otimes \psi)\circ\Delta \in \Hom(\bQ\langle \LL_z\rangle,R)$, where $m_R$ denotes the multiplication in $R$, corresponds to the element $(\Xi_n)_{n\geq 0}$ defined by
\[
\Xi_n(X_1,\ldots,X_n)=\sum_{i=0}^n\Phi(X_1,\ldots,X_i)\,\Psi_{n-i}(X_{i+1},\ldots,X_n).
\]
Now define an automorphism
\[
\begin{aligned}
\sigma: \prod_{n\geq 0}R\llbracket X_1,\ldots,X_n\rrbracket&\rightarrow \prod_{n\geq 0}R\llbracket X_1,\ldots,X_n\rrbracket\\
(\Phi_n)_{n\geq 0}&\mapsto (\sigma(\Phi_n))_{n\geq 0},
\end{aligned}
\]
where $\sigma(\Phi_n)(X_1,\ldots,X_n):=\Phi_n(X_1+\ldots+X_n,X_1+\ldots+X_{n-1},\ldots,X_1)$. Via the isomorphism~\eqref{eqn:trafo}, it induces an automorphism $\sigma^*: \Hom(\bQ\langle \LL_z\rangle,R)\rightarrow \Hom(\bQ\langle \LL_z\rangle,R)$.

\subsubsection{The extended double shuffle relations}
Given a $\bQ$-linear morphism $\varphi \in \Hom(\bQ\langle \LL_z\rangle,R)$, we define a new morphism $\varphi_{\rm corr} \in \Hom(\bQ\langle \LL_z\rangle,R)$ by sending a word $\sy_{k_1}\ldots\sy_{k_n}$ to its coefficient in the power series
\[
\exp\left(\sum_{n=2}^{\infty}\frac{(-1)^n}{n}\varphi(\sy_n)\cdot \sy_1\right) \in R\langle\!\langle \sy_1\rangle\!\rangle.
\]
\begin{definition}\label{def:regdsh}
We say that $\varphi$ satisfies the \emph{extended double shuffle relations}, if, for all words $w_1,w_2 \in \LL_z^*$, we have that:
\begin{itemize}\itemsep3pt
\item[(i)] $\varphi(w_1\ast w_2)=\varphi(w_1)\,\varphi(w_2)$;
\item[(ii)] $\sigma^*(\varphi_{\rm corr}\star\varphi)(w_1\ishuffle w_2)=\sigma^*(\varphi_{\rm corr}\star\varphi)(w_1) \, \sigma^*(\varphi_{\rm corr}\star\varphi)(w_2)$.
\end{itemize}
\end{definition}
\begin{remark}\label{rem:racinetikz}
This definition is equivalent to Racinet's definition (\cite{R}) in the case where $\varphi(\sy_1)=0$, a condition which is, in fact, part of Racinet's definition. The key point is that, by~\cite[Proposition~8]{I}, $\sigma^*(\varphi_{\rm corr}\star \varphi)$ is a homomorphism for the shuffle product if and only if the map
\[
(\varphi_{\rm corr}\star\varphi)\circ\pi_{\LL_z}: \fH^1\rightarrow R
\]
is a homomorphism for the shuffle product on $\fH^1=\bQ+\bQ\langle x,y \rangle y\mspace{1mu}$. Here $\pi_{\LL_z}: \fH^1\rightarrow \bQ\langle \LL_z\rangle$ is the $\bQ$-linear isomorphism defined by mapping $x^{k_1-1} y\cdots x^{k_n-1}y$ to $\sy_{k_1}\ldots \sy_{k_n}\mspace{1mu}$. Then one checks, by the usual duality of the underlying Hopf algebras, that this is also equivalent to $\varphi$ having the EDS property as defined in~\cite[Definition~1]{IKZ}.
\end{remark}

\subsection{Commutative power series}

In this section, we fix a $\bQ$-algebra $R$.
\subsubsection{Preliminaries}
Let
\[
\cB(R):=\prod_{n=0}^{\infty}R\llbracket X_1,\ldots,X_n,Y_1,\ldots,Y_n \rrbracket
\]
be the $R$-module of tuples of formal power series in an increasing number of variables, where $R^\times$ denotes the group of units of $R$ (cf., Equation~\eqref{eq:bimould}). It is a noncommutative, associative, unital $R$-algebra with the concatenation product
\[
\begin{aligned}
\odot: \cB(R)\times \cB(R)&\rightarrow \cB(R)\\
((F_n),(G_n)) &\mapsto (H_n),
\end{aligned}
\]
where 
\[H_n\bi{X_1,\ldots,X_n}{Y_1,\ldots,Y_n} \df \sum_{i=0}^nF_i\bi{X_1,\ldots,X_i}{Y_1,\ldots,Y_i}\,G_{n-i}\bi{X_{i+1},\ldots,X_n}{Y_{i+1},\ldots,Y_n} \qquad (n\geq 0).\]
The unit element is the tuple $\mathbbm{1}:=(1,0,0,\ldots)$ and $F\in \cB(R)$ has an inverse $F^{\odot -1}$ for $\odot$ if and only if $F_0 \in R^\times$, in which case $F^{\odot -1}$ can be defined recursively by
\[
F^{\odot -1}_0=F_0^{-1}, \qquad F^{\odot -1}_n= \sum_{i=0}^{n-1}(-1)^{i+1}\frac{F_{n-i}}{F_0}F^{\odot -1}_i.
\]
Note that there is an embedding
\begin{align}\label{eq:tilde}
\lambda: R\llbracket t\rrbracket&\rightarrow \cB(R) \nonumber\\
P(t)=\sum_{n\geq 0}a_nt^n &\mapsto \lambda(P):=(a_n)_{n\geq 0}
\end{align}
of $R$-algebras. In other words, every power series in one variable can be viewed as an element of~$\cB(R)$.

The algebra $\cB(R)$ is equipped with an involution (the \emph{swap}, cf., Definition~\ref{def:swap})
\[
\begin{aligned}
\sigma: \cB(R)&\rightarrow \cB(R)\\
(F_n)&\mapsto (\sigma(F_n)),
\end{aligned}
\]
where 
\[\sigma(F_n)\bi{X_1,\ldots,X_n}{Y_1,\ldots,Y_n}:=F_n\bi{Y_1+\ldots+Y_n,Y_1+\ldots+Y_{n-1},\ldots,Y_1}{X_n,X_{n-1}-X_n,\ldots,X_1-X_2}. \]
The space~$\cB(R)$ contains two distinguished $R$-submodules
\[
\cM_X(R):= \prod_{n=0}^{\infty}R\llbracket X_1,\ldots,X_n\rrbracket, \qquad \cM_Y(R):= \prod_{n=0}^{\infty}R\llbracket Y_1,\ldots,Y_n\rrbracket,
\]
which clearly are both subalgebras.
\begin{proposition} \label{prop:sigmaanticommutes}
For all $F\in \cM_Y(R)$ and $G \in \cM_X(R)$, we have
\[
\sigma(F\odot G)=\sigma(G)\odot \sigma(F).
\]
\end{proposition}
\begin{proof}
This is a straightforward computation. For every $n\geq 0$, we have on one hand
\[
\sigma((F\odot G)_n)=\sum_{i=0}^nF_i(X_n,\ldots,X_{n-i+1}-X_{n-i+2})\,G_{n-i}(Y_1+\ldots+Y_{n-i},\ldots,Y_1)
\]
and on the other hand
\[
\begin{aligned}
(\sigma(G)\odot\sigma(F))_n&=\sum_{i=0}^nG_i(Y_1+\ldots+Y_i,\ldots,Y_1)\,F_{n-i}(X_n,\ldots,X_{i+1} - X_{i+2})\\
&=\sum_{i=0}^nG_{n-i}(Y_1+\ldots+Y_{n-i},\ldots,Y_1)\,F_i(X_n,\ldots,X_{n-i+1}-X_{n-i+2}),
\end{aligned}
\]
which are clearly equal, since $R$ is commutative.
\end{proof}
\subsubsection{A bijection}

Now define a map
\[
\begin{aligned}
\Phi: \cM_X(R)^{\times}&\rightarrow \cB(R)^{\times}\\
F&\mapsto \sigma(F)\odot c(F)^{\odot -1}\odot F,
\end{aligned}
\]
where $c: \cM_X(R)\rightarrow \prod_{n=0}^{\infty}R$ is the map which in each component sends a power series to its constant term. In the other direction, define
\[
\begin{aligned}
\Psi: \cB(R)^\times&\rightarrow \cM_X(R)^\times\\
F&\mapsto c_X(F),
\end{aligned}
\]
where $c_X: \cB(R)\rightarrow \cM_X(R)$ is the map which in each component sets $Y_1=\ldots=Y_n=0$.
\begin{proposition} \label{prop:bijectionv1}
The maps $\Phi$ and $\Psi$ define mutually inverse bijections between $\cM_X(R)^\times$ and the subset of $\cB(R)^\times$ consisting of those elements $F$ which satisfy the following two conditions:
\begin{itemize}
\item[(i)] $\sigma(F)=F$;
\item[(ii)] There exist elements $G \in \cM_Y(R)^\times$ and $H\in \cM_X(R)^\times$ such that $c(G)=\mathbbm{1}$ and such that $F=G\odot H$.
\end{itemize}
\end{proposition}
\begin{proof}
We begin by proving that $\Phi(F)$ satisfies conditions (i) and (ii) above, for every $F\in \cM_X(R)^\times$. For condition (i), this follows directly from Proposition~\ref{prop:sigmaanticommutes} since $\sigma(F)\in \cM_Y(R)$. For condition (ii), it is clear that $\Phi(F)=G\odot H$ for $H=F \in \cM_X(R)^\times$ and $G=\sigma(F)\odot c(F)^{\odot -1} \in \cM_Y(R)^\times$. Moreover, $c(G)=c(\sigma(F))\odot c(F)^{\odot -1}=\mathbbm{1}$, since $c$ is an algebra homomorphism for $\odot$ and $c(\sigma(F))=c(F)$. Hence $\Phi(F)$ satisfies conditions (i) and (ii) of the proposition.

We next prove that $\Phi$ and $\Psi$ are mutually inverse. Indeed, we have
\[
\Psi(\Phi(F))=c_X(\sigma(F)\odot c(F)^{\odot -1}\odot F)=c(\sigma(F))\odot c(F)^{-1}\odot F=F,
\]
since $c_X$ agrees with $c$ on the subspace $\cM_Y(R)\subset \cB(R)$ and is the identity on $\cM_X(R)$. On the other hand, since $c(G)=\mathbbm{1}$, we have
\[
\Phi(\Psi(G\odot H))=\Phi(H)=\sigma(H)\odot c(H)^{\odot -1}\odot H.
\]
In order to prove that this equals $G\odot H$, since $H$ is invertible, it suffices to verify that $G=\sigma(H)\odot c(H)^{\odot -1}$. Indeed, since $G\odot H$ is $\sigma$-invariant, we have
\[
G\odot H=\sigma(G\odot H)=\sigma(H)\odot \sigma(G),
\]
by Proposition~\ref{prop:sigmaanticommutes}. Applying $c_X$ to both sides of this equality, we see that $H=c(H)\odot \sigma(G)$, and multiplying both sides with $c(H)^{\odot -1}$, then applying $\sigma$ yields the desired equality $G=\sigma(H)\odot c(H)^{\odot -1}$.
\end{proof}

\subsubsection{Hopf algebra structure}

We next equip $\cB(R)$ with the structure of a complete Hopf algebra over $R$.
To begin with, define a homomorphism of $R$-algebras by $(k\geq 1, d\geq 0 \text{  and } (k,d)\neq (1,0))$
\[
\begin{aligned}
\Delta_{\ast}: R[X,Y]&\rightarrow R[X,Y]\otimes_R R[X,Y]\\
X^{k-1}Y^d&\mapsto  X^{k-1}Y^d \otimes 1+1\otimes X^{k-1}Y^d\\
&\qquad +\sum_{\substack{k_1,k_2\geq 1 \\k_1+k_2=k }}\sum_{\substack{d_1,d_2\geq 0 \\ d_1+d_2=d}}\binom{d_1+d_2}{d_1}\,X^{k_1-1}Y^{d_1}\otimes X^{k_2-1}Y^{d_2}.
\end{aligned}
\]
% \hen{Maybe more natural (more ugly?)
% \[
% \begin{aligned}
% \Delta_{\ast}: R[X,Y]&\rightarrow R[X,Y]\otimes_R R[X,Y]\\
% X^{k-1}\frac{Y^d}{d!}&\mapsto  X^{k-1}\frac{Y^d}{d!} \otimes 1+1\otimes X^{k-1}\frac{Y^d}{d!} +\sum_{\substack{k_1,k_2\geq 1\\ d_1,d_2\geq 0 \\k_1+k_2=k \\ d_1+d_2=d}}  X^{k_1-1}\frac{Y^{d_1}}{d_1!}\otimes X^{k_2-1}\frac{Y^{d_2}}{d_2!}.
% \end{aligned}
% \]

% }
Since $\Delta_{\ast}$ respects the degree of monomials, it uniquely extends to an $R$-algebra homomorphism $\widehat{\Delta}_{\ast}: R\llbracket X,Y\rrbracket\rightarrow R\llbracket X,Y\rrbracket\,\widehat{\otimes}_R\, R\llbracket X,Y\rrbracket$, where $\widehat{\otimes}$ denotes the completed tensor product. Finally, it extends to an $R$-linear map
\[
\widehat{\Delta}_{\ast}: \cB(R)\rightarrow \cB(R)\;\widehat{\otimes}\;\cB(R)
\]
by the requirement that $\Delta_{\ast}$ is an algebra homomorphism for the concatenation product $\odot$. The triple $(\cB(R),\odot,\widehat{\Delta}_{\ast})$ is then a complete bialgebra, and one can show that it is in fact a Hopf algebra. We let $\cG(R) \subset \cB(R)$ denote the subspace of group-like elements, i.e., elements $F=(F_n)_{n\geq 0}$ which satisfy $\widehat{\Delta}_*(F)=F\otimes F$ and $F_0=1$. In fact, $\cG(R)\subset \cB(R)$ is a subgroup for the concatenation product.

Our next goal is to state and prove a variant of Proposition~\ref{prop:bijectionv1} for group-like elements. First of all, the map $\Psi: \cB(R)^\times\rightarrow \cM_X(R)^\times$ which sends all variables $Y_i=0$ respects the subspace of group-like elements, hence induces a map
\[
\Psi: \cG(R)\rightarrow \cG_X(R),
\]
where $\cG_X(R)$ denotes the group-like elements of the Hopf subalgebra $\cM_X(R)\subset \cB(R)$. Hence, if in addition $F=G\odot H$ with $G,H$ as in Proposition~\ref{prop:bijectionv1}, then $H \in \cG_X(R)$. We next state a refinement of this result.
\begin{proposition}\label{prop:appendix}
Let $F \in \cB(R)^\times$ satisfy conditions (i) and (ii) of Proposition~\ref{prop:bijectionv1}. Then the following are equivalent:
\begin{enumerate}
\item[{\upshape(i)}] $F \in \cG(R)$;
\item[{\upshape(ii)}] $H\in \cG_X(R)$, and $\sigma(\lambda(F_{\rm corr})\odot c_X(F)) \in \cG_Y(R)$, where $\lambda$ is defined by~\eqref{eq:tilde}, 
\[
F_{\rm corr}:=\exp\left(\sum_{n=2}^{\infty}\frac{(-1)^n}{n}f_{n,0}\, t^n  \right),
\]
and $f_{n,0}$ is the coefficient of $X^{n-1}$ in $F_1 \in R\llbracket X,Y\rrbracket$.
\end{enumerate}
\end{proposition}
\begin{proof}
We first prove that $\text{\upshape{(i)}}\Rightarrow \text{\upshape{(ii)}}$. By the above discussion, we already know that $H\in \cG_X(R) \subset \cG(R)$. Therefore, since $\cG(R)$ is a group, we also know that $G \in \cG_Y(R)$. Moreover, by Proposition~\ref{prop:bijectionv1}, we have
\[
G=\sigma(F)\odot c(F)^{\odot -1}=\sigma(F)\odot \lambda(F_{\rm corr})\odot \lambda(\exp(-f_{1,0}t)),
\]
where the last equality follows from~\cite[Equation~(32)]{HI}. Note that $\lambda(\exp(-f_{1,0}t)) \in \cG(R)$. Hence, using Proposition~\ref{prop:sigmaanticommutes} and again the fact that $\cG(R)$ is a group, we see that $\sigma(\lambda(F_{\rm corr})\odot F) \in \cG_Y(R)$. Hence, also $\sigma(\lambda(F_{\rm corr})\odot c_X(F)) \in \cG_Y(R)$ as was to be shown.

Conversely, if both $H=c_X(F)$ and $\sigma(\lambda(F_{\rm corr})\odot c_X(F))$ are group-like, then so is $G\odot H$, since
\[
G\odot H =\sigma(c_X(F))\odot c(F)^{\odot -1}\odot c_X(F)
=\sigma(\lambda(F_{\rm corr})\odot c_X(F))\odot \lambda(\exp(-f_{1,0}t))\odot c_X(F)
\]
is a product of group-like elements, where the first equality follows from Proposition~\ref{prop:bijectionv1}.
\end{proof}

\end{document}

%% file: Diagrams/formal_picture.tex
\begin{tikzpicture}[>=stealth, baseline=(current bounding box.center)]
% Nodes  

\node (zero) at (-4,2.5) {$0$};
%\node (mzeta) at (-2,2) {$\Q[\fzeta(4),\fzeta(6)]$}; # looks too big
%\node (mzeta) at (-2,2) {$\Q[\fzeta(2)]$};
\node (zeta) at (0,2.5) {$\Q[\fzeta(2)]$};
\node (mzv) at (2,2.5) {$\fmz$};

\node (S) at (-4,4) {$\fs$};
\node (mf) at (-2,4) {$\fmf$};
\node (qmf) at (0,4) {$\fqmf$};
\node (E) at (2,4) {$\fames$};
\node[gray] at (2.3,4.5) {?};
\node (Fil) at (4,4) {$\fil{\lwt}{0}{\fmes}$};
\node (fmes) at (6,4) {$\fmes$};
\node[gray] at (5.2,4.65) {{\footnotesize ?}};
\node[gray] at (5.2,4.3) {$\cong$};

% Loop Arrow for the action of sl2 on fmes
\path[->] (qmf) edge [in=70, out=110, looseness=2, loop above] node {$\sltwo$} ();
\path[->,dashed,gray] (E) edge [in=70, out=110, looseness=2, loop above] node {$\sltwo$} ();
\path[->] (Fil) edge [in=70, out=110, looseness=2, loop above] node {$\sltwo$} ();
\path[->] (fmes) edge [in=70, out=110, looseness=2, loop above] node {$\sltwo$} ();

% Inclusions
\draw[right hook->] (S) -- (mf);
\draw[right hook->] (mf) -- (qmf);
\draw[right hook->] (qmf) -- (E);
\draw[right hook->] (E) -- (Fil);
\draw[right hook->] (Fil) -- (fmes);

\draw[right hook->] (zero) -- (zeta); 
\draw[right hook->] (zeta) -- (mzv); 

\draw[->] (S) -- (zero) node[midway, left] {$\pi$}; 
\draw[->] (qmf) -- (zeta); 
\draw[->] (E) -- (mzv); 
\draw[->] (Fil) -- (mzv); 
\draw[->] (fmes) -- (mzv); % node[midway, below] {$\pi$}; 

\end{tikzpicture}

%% file: Diagrams/classical_picture.tex
\begin{tikzpicture}[>=stealth, baseline=(current bounding box.center)]
% Nodes  
\node (zero) at (-4,2.5) {$0$};
%\node (mzeta) at (-2,2) {$\Q[\fzeta(4),\fzeta(6)]$};
%\node (mzeta) at (-2,2) {$\Q[\fzeta(2)]$};
\node (zeta) at (0,2.5) {$\Q[\zeta(2)]$};
\node (mzv) at (2,2.5) {$\mz$};

\node (S) at (-4,4) {$\mathbb{S}$};
\node (mf) at (-2,4) {$\mathbb{M}$};
\node (qmf) at (0,4) {$\widetilde{\mathbb{M}}$};
\node (E) at (2,4) {$\ames$};
\node[gray] (Fil) at (4,4) {$\ames^{\rm{reg}}$};
\node[gray] (fmes) at (6,4) {$?$};
 \node at (7,4) {$\subset \mathcal{O}(\Ha)$}; 
\node[gray] at (2.3,4.5) {?};
\node[gray] at (4.3,4.5) {?};

\path[->] (qmf) edge [in=70, out=110, looseness=2, loop above] node {$\sltwo$} ();
\path[->,dashed,gray] (E) edge [in=70, out=110, looseness=2, loop above] node {$\sltwo$} ();
\path[->,dashed,gray] (Fil) edge [in=70, out=110, looseness=2, loop above] node {$\sltwo$} ();
% Arrows
\draw[right hook->] (S) -- (mf);
\draw[right hook->] (mf) -- (qmf);
\draw[right hook->] (qmf) -- (E);
\draw[right hook->,gray] (E) -- (Fil);
\draw[right hook->, dashed, gray] (Fil) -- (fmes);

\draw[right hook->] (zero) -- (zeta); 
\draw[right hook->] (zeta) -- (mzv); 

\draw[->] (S) -- (zero)  node[midway, left, align=center] {\footnotesize const.\ \\[-3pt]\footnotesize term\ };
\draw[->] (qmf) -- (zeta); 
\draw[->] (E) -- (mzv); 
\draw[->,gray] (Fil) -- (mzv); 
\draw[->,dashed,gray] (fmes) -- (mzv); 
\end{tikzpicture}

%% file: Diagrams/rational_picture.tex
\begin{tikzpicture}[>=stealth, baseline=(current bounding box.center)]
% Nodes  
\node (zero) at (-4,2.5) {$0$};
%\node (mzeta) at (-2,2) {$\Q[\fzeta(4),\fzeta(6)]$};
%\node (mzeta) at (-2,2) {$\Q[\fzeta(2)]$};
\node (zeta) at (0,2.5) {$\Q$};
\node (mzv) at (2,2.5) {$\Q$};

\node (S) at (-4,4) {$\mathcal{S}$};
\node (mf) at (-2,4) {$\mf$};
\node (qmf) at (0,4) {$\qmf$};
\node (E) at (2,4) {$\ames_\Q$};
\node[gray] at (2.3,4.58) {?};
\node (Fil) at (4,4) {$\mz^{\circ}_q$};
\node[gray] at (4.32,4.58) {?};
\node (fmes) at (6,4) {$\mz_q$}; %
\node[gray] at (6.32,4.58) {?};
\node at (7,4) {$\subset \Q\llbracket q\rrbracket$};
\node[gray] at (5.1,4.65) {{\footnotesize ?}};
\node[gray] at (5.1,4.3) {$\cong$};

% Loop Arrow for the action of sl2 on fmes
\path[->] (qmf) edge [in=70, out=110, looseness=2, loop above] node {$\sltwo$} ();
\path[->,dashed,gray] (E) edge [in=70, out=110, looseness=2, loop above] node {$\sltwo$} ();
\path[->,dashed,gray] (Fil) edge [in=70, out=110, looseness=2, loop above] node {$\sltwo$} ();
\path[->,dashed,gray] (fmes) edge [in=70, out=110, looseness=2, loop above] node {$\sltwo$} ();

% Arrows
\draw[right hook->] (S) -- (mf);
\draw[right hook->] (mf) -- (qmf);
\draw[right hook->] (qmf) -- (E);
\draw[right hook->] (E) -- (Fil);
\draw[right hook->] (Fil) -- (fmes);

\draw[right hook->] (zero) -- (zeta); 
%\draw[right hook->] (mzeta) -- (zeta); 
\draw[right hook->] (zeta) -- (mzv); 

\draw[->] (S) -- (zero) node[midway, left, align=center] {\footnotesize const.\ \\[-3pt]\footnotesize term\ };
\draw[->] (qmf) -- (zeta); 
\draw[->] (E) -- (mzv); 
\draw[->] (Fil) -- (mzv); 
\draw[->] (fmes) -- (mzv); 
\end{tikzpicture}

%% file: Diagrams/partition.tex
\begin{tikzpicture}[scale=0.5]

\draw (0,2.5) -- (0,0) -- (1,0) -- (1,1.5) -- (2,1.5) -- (2,2.5);
\draw [densely dotted] (0,2.5) -- (0,3.5);
\draw [densely dotted] (2,2.5) -- (3,2.5) -- (3,3.5);
\draw (3,3.5) -- (4,3.5) -- (4,4.5) -- (6.2,4.5) -- (6.2,5.5) -- (0,5.5) -- (0,3.5);

\draw [thick, red] (0,5.5) -- (6.2,5.5); 
\draw [thick, blue] (0,5.5) -- (0,3.5); 
\draw [thick, blue, densely dotted] (0,2.5) -- (0,3.5);
\draw [thick, blue] (0,0) -- (0,2.5);

\draw[decoration={brace,raise=2pt},decorate] (0,5.5) -- node[above=2pt] {\scriptsize $X_1$} (6.2,5.5);
\draw[decoration={brace,raise=2pt},decorate] (0,4.5) -- node[above=2pt] {\scriptsize $X_2$} (4,4.5);
\draw[decoration={brace,raise=2pt},decorate] (0,2.5) -- node[above=2pt] {\scriptsize $X_{r-1}$} (2,2.5);
\draw[decoration={brace,raise=2pt},decorate] (0,1.5) -- node[above=2pt] {\scriptsize $X_{r}$} (1,1.5);

\draw[decoration={brace,raise=2pt},decorate] (6.2,5.5) -- node[right=2pt] {\scriptsize $Y_{1}$} (6.2,4.5);
\draw[decoration={brace,raise=2pt},decorate] (4,4.5) -- node[right=2pt] {\scriptsize $Y_{2}$} (4,3.5);
\draw[decoration={brace,raise=2pt},decorate] (2,2.5) -- node[right=2pt] {\scriptsize $Y_{r-1}$} (2,1.5);
\draw[decoration={brace,raise=2pt},decorate] (1,1.5) -- node[right=2pt] {\scriptsize $Y_{r}$} (1,0);

\draw [lightgray,ultra thin] (4,5.5) -- (4,4.7);
\draw [lightgray,ultra thin] (4.3,5.5) -- (4.3,4.7);
\draw [lightgray,ultra thin] (4.6,5.5) -- (4.6,4.7);
\draw [lightgray,ultra thin] (4.9,5.5) -- (4.9,4.7);
\draw [lightgray,ultra thin] (5.2,5.5) -- (5.2,4.7);
\draw [lightgray,ultra thin] (5.5,5.5) -- (5.5,4.7);
\draw [lightgray,ultra thin] (5.8,5.5) -- (5.8,4.7);
\draw [lightgray,ultra thin] (3.7,5.5) -- (3.7,3.7);
\draw [lightgray,ultra thin] (3.4,5.5) -- (3.4,3.7);
\draw [lightgray,ultra thin] (3.1,5.5) -- (3.1,3.7);
\draw [lightgray,ultra thin] (2.8,5.5) -- (2.8,2.7);
\draw [lightgray,ultra thin] (2.5,5.5) -- (2.5,2.7);
\draw [lightgray,ultra thin] (2.2,4.7) -- (2.2,2.7);
\draw [lightgray,ultra thin] (1.9,4.7) -- (1.9,2.7);
\draw [lightgray,ultra thin] (1.6,4.7) -- (1.6,1.7);
\draw [lightgray,ultra thin] (1.3,5.5) -- (1.3,3.1);
\draw [lightgray,ultra thin] (1,5.5) -- (1,3.3);
\draw [lightgray,ultra thin] (0.7,5.5) -- (0.7,3.3);
\draw [lightgray,ultra thin] (0.4,5.5) -- (0.4,2.4);
\draw [lightgray,ultra thin] (1.3,2.7) -- (1.3,1.6);
\draw [lightgray,ultra thin] (1,2.7) -- (1,1.5);
\draw [lightgray,ultra thin] (0.7,2.7) -- (0.7,2);
\draw [lightgray,ultra thin] (0.4,1.7) -- (0.4,0);
\draw [lightgray,ultra thin] (0.7,1.7) -- (0.7,0);

\draw [lightgray,ultra thin] (0.1,0.3) -- (0.9,0.3);
\draw [lightgray,ultra thin] (0.1,0.6) -- (0.9,0.6);
\draw [lightgray,ultra thin] (0.1,0.9) -- (0.9,0.9);
\draw [lightgray,ultra thin] (0.1,1.2) -- (0.9,1.2);
\draw [lightgray,ultra thin] (0.1,1.5) -- (0.9,1.5);

\draw [lightgray,ultra thin] (0.9,1.9) -- (1.9,1.9);
\draw [lightgray,ultra thin] (0.9,2.2) -- (1.9,2.2);
\draw [lightgray,ultra thin] (0.1,2.5) -- (1.9,2.5);

\draw [lightgray,ultra thin] (0.1,3.5) -- (2.9,3.5);
\draw [lightgray,ultra thin] (1.1,3.2) -- (2.9,3.2);
\draw [lightgray,ultra thin] (1.8,2.9) -- (2.9,2.9);

\draw [lightgray,ultra thin] (0.1,4.5) -- (3.9,4.5);
\draw [lightgray,ultra thin] (0.1,4.16) -- (3.9,4.16);
\draw [lightgray,ultra thin] (0.1,3.85) -- (3.9,3.85);

\draw [lightgray,ultra thin] (0.1,4.85) -- (1.4,4.85);
\draw [lightgray,ultra thin] (0.1,5.2) -- (1.4,5.2);
\draw [lightgray,ultra thin] (2.3,5.2) -- (6,5.2);
\draw [lightgray,ultra thin] (2.3,4.85) -- (6,4.85);

\draw [->,thick] (7.4,3) -- node[above=2pt]{\tiny conjugate} (10.4,3);

\begin{scope}[shift={(12,0.3)}]

\draw (0,2.5) -- (0,-0.6) -- (1,-0.6) -- (1,1.5) -- (2,1.5) -- (2,2.5);
\draw [densely dotted] (0,2.5) -- (0,3.5);
\draw [densely dotted] (2,2.5) -- (3,2.5) -- (3,3.5);
\draw (3,3.5) -- (4,3.5) -- (4,4.5) -- (5.5,4.5) -- (5.5,5.5) -- (0,5.5) -- (0,3.5);

\draw [thick, blue] (0,5.5) -- (5.5,5.5); 
\draw [thick, red] (0,5.5) -- (0,3.5); 
\draw [thick, red, densely dotted] (0,2.5) -- (0,3.5);
\draw [thick, red] (0,-0.6) -- (0,2.5);

\draw [lightgray,ultra thin] (4,5.5) -- (4,4.7);
\draw [lightgray,ultra thin] (4.3,5.5) -- (4.3,4.7);
\draw [lightgray,ultra thin] (4.6,5.5) -- (4.6,4.7);
\draw [lightgray,ultra thin] (4.9,5.5) -- (4.9,4.7);
\draw [lightgray,ultra thin] (5.2,5.5) -- (5.2,4.7);
\draw [lightgray,ultra thin] (3.7,5.5) -- (3.7,3.7);
\draw [lightgray,ultra thin] (3.4,5.5) -- (3.4,3.7);
\draw [lightgray,ultra thin] (3.1,5.5) -- (3.1,3.7);
\draw [lightgray,ultra thin] (2.8,5.5) -- (2.8,2.7);
\draw [lightgray,ultra thin] (2.5,5.5) -- (2.5,2.7);
\draw [lightgray,ultra thin] (2.2,4.7) -- (2.2,2.7);
\draw [lightgray,ultra thin] (1.9,4.7) -- (1.9,2.7);
\draw [lightgray,ultra thin] (1.6,4.7) -- (1.6,1.7);
\draw [lightgray,ultra thin] (1.3,5.5) -- (1.3,3.1);
\draw [lightgray,ultra thin] (1,5.5) -- (1,3.3);
\draw [lightgray,ultra thin] (0.7,5.5) -- (0.7,3.3);
\draw [lightgray,ultra thin] (0.4,5.5) -- (0.4,2.4);
\draw [lightgray,ultra thin] (1.3,2.7) -- (1.3,1.6);
\draw [lightgray,ultra thin] (1,2.7) -- (1,1.5);
\draw [lightgray,ultra thin] (0.7,2.7) -- (0.7,2);
\draw [lightgray,ultra thin] (0.4,1.7) -- (0.4,-0.5);
\draw [lightgray,ultra thin] (0.7,1.7) -- (0.7,-0.5);
\draw [lightgray,ultra thin] (0.1,-0.3) -- (0.9,-0.3);
\draw [lightgray,ultra thin] (0.1,0) -- (0.9,0);
\draw [lightgray,ultra thin] (0.1,0.3) -- (0.9,0.3);
\draw [lightgray,ultra thin] (0.1,0.6) -- (0.9,0.6);
\draw [lightgray,ultra thin] (0.1,0.9) -- (0.9,0.9);
\draw [lightgray,ultra thin] (0.1,1.2) -- (0.9,1.2);
\draw [lightgray,ultra thin] (0.1,1.5) -- (0.9,1.5);

\draw [lightgray,ultra thin] (0.9,1.9) -- (1.9,1.9);
\draw [lightgray,ultra thin] (0.9,2.2) -- (1.9,2.2);
\draw [lightgray,ultra thin] (0.1,2.5) -- (1.9,2.5);

\draw [lightgray,ultra thin] (0.1,3.5) -- (2.9,3.5);
\draw [lightgray,ultra thin] (1.1,3.2) -- (2.9,3.2);
\draw [lightgray,ultra thin] (1.8,2.9) -- (2.9,2.9);

\draw [lightgray,ultra thin] (0.1,4.5) -- (3.9,4.5);
\draw [lightgray,ultra thin] (0.1,4.16) -- (3.9,4.16);
\draw [lightgray,ultra thin] (0.1,3.85) -- (3.9,3.85);

\draw [lightgray,ultra thin] (0.1,4.85) -- (1.4,4.85);
\draw [lightgray,ultra thin] (0.1,5.2) -- (1.4,5.2);
\draw [lightgray,ultra thin] (2.3,5.2) -- (5.4,5.2);
\draw [lightgray,ultra thin] (2.3,4.85) -- (5.4,4.85);

\draw[decoration={brace,raise=2pt},decorate] (0,5.5) -- node[above=2pt] {\scriptsize $Y_1 + \dots + Y_r$} (5.5,5.5);
\draw[decoration={brace,raise=2pt},decorate] (0,4.4) -- node[above=2pt] {\tiny $\,\,Y_1 + \dots + Y_{r-1}$} (4,4.4);
\draw[decoration={brace,raise=2pt},decorate] (0,2.5) -- node[above=2pt] {\tiny $\,Y_1+Y_2$} (2,2.5);
\draw[decoration={brace,raise=2pt},decorate] (0,1.5) -- node[above=2pt] {\tiny $Y_1$} (1,1.5);

\draw[decoration={brace,raise=2pt},decorate] (5.5,5.5) -- node[right=2pt] {\scriptsize $X_r$} (5.5,4.5);
\draw[decoration={brace,raise=2pt},decorate] (4,4.5) -- node[right=2pt] {\scriptsize $X_{r-1}-X_r$} (4,3.5);
\draw[decoration={brace,raise=2pt},decorate] (2,2.5) -- node[right=2pt] {\scriptsize $X_2-X_3$} (2,1.5);
\draw[decoration={brace,raise=2pt},decorate] (1,1.5) -- node[right=2pt] {\scriptsize $X_1-X_2$} (1,-0.6);

\end{scope}
\end{tikzpicture}